\documentclass[a4paper,12pt]{amsart}

\usepackage{graphics,graphicx}
\usepackage{tikz-cd}
\usetikzlibrary{calc}

\usepackage[utf8]{inputenc}
\usepackage[T1]{fontenc}

\usepackage{amssymb,stix}

\usepackage{xfrac}
\usepackage{booktabs}

\newtheorem{thm}{Theorem}[section]
\newtheorem{cor}[thm]{Corollary}
\newtheorem{lem}[thm]{Lemma}
\newtheorem{prop}[thm]{Proposition}

\theoremstyle{definition}

\newtheorem{rem}[thm]{Remark}
\newtheorem{exa}[thm]{Example}
\newtheorem{nota}[thm]{Notation}

\DeclareMathOperator{\ev}{ev}
\DeclareMathOperator{\Hess}{Hess}
\DeclareMathOperator{\inj}{inj}
\DeclareMathOperator{\mo}{mod}
\DeclareMathOperator{\Ric}{Ric}
\DeclareMathOperator{\SL}{SL}
\DeclareMathOperator{\Sp}{Sp}
\DeclareMathOperator{\SO}{SO}
\DeclareMathOperator{\GL}{GL}
\DeclareMathOperator{\PGL}{PGL}
\DeclareMathOperator{\tr}{tr}
\DeclareMathOperator{\Vol}{Vol}
\DeclareMathOperator{\barDH}{Bar}

\newcommand{\cA}{\mathcal{A}}
\newcommand{\fa}{\mathfrak{a}}
\newcommand{\ta}{\tilde \alpha}
\newcommand{\ha}{\hat{\alpha}}

\newcommand{\Roots}{\hat{R}} 
\newcommand{\rRoots}{R}
\newcommand{\opP}{\mathcal{P}}
\newcommand{\rhodeux}{\varrho^{(2)}}
\newcommand{\rhoun}{\varrho^{(1)}}

\begin{document}

\title{Ricci flat Kähler metrics on rank two complex symmetric spaces}

\author{Olivier Biquard \and Thibaut Delcroix}
\address{Sorbonne Université and École Normale Supérieure, PSL University}
\email{olivier.biquard@ens.fr}
\address{École Normale Supérieure, PSL University}
\email{thibaut.delcroix@ens.fr}

\thanks{This work has received support under the program ``Investissements d'Avenir'' launched by the French Government and implemented by ANR with the reference ANR-10-IDEX-0001-02 PSL}

\date{}

\begin{abstract}
  We obtain Ricci flat Kähler metrics on complex symmetric spaces of
  rank two by using an explicit asymptotic model whose geometry at
  infinity is interpreted in the wonderful compactification of the
  symmetric space. We recover the metrics of Biquard-Gauduchon in the
  Hermitian case and obtain in addition several new metrics.
\end{abstract}

\maketitle

\section{Introduction}

A (complex) symmetric space is a homogeneous space under a complex
semisimple Lie group, whose isotropy Lie subalgebra is the fixed point
set of a complex involution.  It may always be viewed as a
complexified compact symmetric space, thus also as the tangent or
cotangent bundle of such a compact symmetric space, equipped with the
appropriate complex structure.  Such a complex manifold may admit a
Ricci flat Kähler metric and indeed several such metrics have already
been exhibited: notably Stenzel's metrics on rank one complex
symmetric spaces \cite{Ste93}, and Biquard-Gauduchon's hyperKähler
metrics on Hermitian complex symmetric spaces \cite{BG96}. These metrics are Asymptotically Conical (AC), with smooth cone at infinity for Stenzel's metrics and singular cone for Biquard-Gauduchon's metrics.

Tian and Yau developed in \cite{TiaYau90,TiaYau91} a general method
to obtain complete Ricci flat Kähler metrics on non-compact complex
manifolds by viewing such a manifold as the complement of a smooth
divisor supporting the anticanonical divisor in a Fano manifold (or
more generally orbifold).  If the anticanonical divisor thus obtained
is non-reduced, then a condition has to be imposed on the reduced
divisor, namely that it admits a, necessarily positive,
Kähler-Einstein metric. The Tian-Yau theorem was refined by various authors along the years, and most notably in the AC case by Conlon and Hein \cite{ConHei13,ConHei15}. Recently, new examples of AC Calabi-Yau metrics with singular cone at infinity were constructed in \cite{CDR16,Li17,Sze17}, in particular on $\Bbb{C}^n$ for $n>2$.

In this article we use the Tian-Yau philosophy to produce Ricci flat
Kähler metrics on complex symmetric spaces of rank two by viewing such
a manifold as the open orbit in its wonderful compactification. 
Let $G/H$ denote the symmetric space and $X$ its wonderful compactification. 
The boundary divisor $X\setminus G/H$ is then a simple normal crossing divisor 
with two irreducible components $D_1$ and $D_2$, which supports an anticanonical 
divisor for the wonderful compactification (note that this manifold is not always
Fano \cite{Ruz12}).  Each component divisor is a two-orbits manifold
with one open orbit which is a homogeneous fibration over a
generalized flag manifold with fibers a complex symmetric space.  We
will search for AC metrics with singular cone at infinity obtained by
taking a line bundle over a singular Kähler-Einstein manifold $\check{D}_2$ which is
a blow-down of the boundary divisor $D_2$. We find an ansatz to
desingularize this singular cone using the other boundary divisor $D_1$ and
in particular the Stenzel metric on the fibers of the open orbit of
this other boundary divisor, which gives the desingularization in the
`collapsed directions'. It is justified by analyzing the explicit
examples produced by the first author and Gauduchon with the Kähler
geometry techniques developed by the second author to study
horosymmetric spaces \cite{DelHoro} (as both symmetric spaces and the
open orbits of divisors in their wonderful compactifications are
horosymmetric).  There is no canonical choice of behavior on the
respective divisors: we obtain examples where only one choice works,
and examples where both choices work, thus providing two Ricci-flat
Kähler metrics with different asymptotic behavior.

\begin{thm}
There exists a Ricci flat Kähler metric with the above boundary behavior on the following 
indecomposable rank two symmetric spaces:
\begin{itemize}
\item for one ordering of divisors, on the non-Hermitian symmetric spaces 
\[ \Sp_8/(\Sp_4\times \Sp_4), \qquad G_2/\SO_4, \qquad G_2\times G_2/G_2, \qquad \SO_5\times \SO_5/\SO_5, \] 
\item on each Hermitian symmetric space, there is a Ricci flat Kähler metric for one choice 
of ordering, which corresponds to Biquard-Gauduchon's metrics, 
\item on the following Hermitian symmetric spaces, the other choice of ordering of divisors produces a 
Ricci flat Kähler metric with a different asymptotic cone: 
\[ \SO_n/S(O_2\times O_{n-2})~\text{ for}~n\geq 5, \qquad  
\SL_5/S(\GL_2\times \GL_3). \] 
\end{itemize}
\end{thm}

There remains a number of cases not covered by the theorem, including the simplest 
rank two symmetric space $\SL_3/\SO_3$. The main reason is that the ansatz considered 
degenerates too badly on the divisor $D_1$, so that the usual techniques to produce 
the Ricci flat solution from an asymptotic solution do not apply. We still expect 
that such metrics exist, and we hope to come back to this problem in the future. There are however two exceptions, which are the symmetric 
space $G_2/\SO_4$ and the group $G_2\times G_2/G_2$, in which case we can prove that 
there does not exist any metric with the expected asymptotic behavior for one ordering 
of divisors. 

Indeed, the existence of such a metric requires the existence of a positive Kähler-Einstein 
metric on the singular $\mathbb{Q}$-Fano variety $\check{D}_2$. There is no general existence 
theorem for Kähler-Einstein metrics on singular Fano varieties. 
For our purpose we thus prove the following characterization:

\begin{thm}\label{th:2}
Assume $\check{D}_2$ is the $\mathbb{Q}$-Fano blowdown of a boundary divisor in the 
wonderful compactification of a rank two indecomposable symmetric space, then 
it admits a (singular) Kähler-Einstein metric if and only if the combinatorial 
condition in \cite{DelKSSV} is satisfied, thus if and only if it is K-stable. 
\end{thm}

Since $\check{D}_2$ is a (colored) rank one horosymmetric variety, 
the Kähler-Einstein equation reduces to a one-variable second order ODE.
The proof is nevertheless obtained by using the continuity method, 
in which the main difficulty is the $C^0$-estimate as usual in the positive 
Kähler-Einstein situation. 
It turns out that the obstruction cancels except for one choice of $D_2$ in the 
cases $G_2/\SO_4$ and $G_2\times G_2/G_2$. These examples are thus natural examples
of singular cohomogeneity one $\mathbb{Q}$-Fano varieties with no singular Kähler-Ricci 
solitons. 
We actually prove the last theorem in a more general situation (see Section~\ref{sec:KE}), 
so that it applies to a larger class of rank one horosymmetric varieties, 
and for variants of the Kähler-Einstein equation. 

There is an obvious question of generalizing these results to higher rank symmetric spaces. We expect the general setting to be the same: the wonderful compactification is obtained by adding $r$ divisors, where $r$ is the rank. For each choice of divisor of the compactification one can try to produce a Ricci flat Kähler metric whose asymptotic cone is a line bundle over a singular blowdown of this divisor. The first step is of course to check the same combinatorial condition as in Theorem~\ref{th:2}, which is not obvious. Here the desingularization is encoded in the combinatorics of the divisors of the compactification. This procedure should lead to a maximum of $r$ distinct Kähler Ricci flat metrics on the symmetric space.

The article is organized as follows. In Section~\ref{sec:setup} we
introduce the relevant combinatorial data associated to symmetric
spaces, their wonderful compactifications, and derive from
\cite{DelHoro} the translation of the Ricci flat equation as a real
two-variables Monge-Ampère equation. In Section~\ref{sec:KE}, we state
a numerical criterion of existence of solutions to a one-variable ODE
which arises as the equation ruling the existence of positive
Kähler-Einstein metrics on rank one horosymmetric spaces or simple
variants of this. In the remaining of this section, we determine when
this criterion is satisfied in the case where the equation exactly
encodes the existence of a (singular) Kähler-Einstein metrics on a
colored $\mathbb{Q}$-Fano compactification of the horosymmetric spaces
arising as the boundary divisors in a wonderful compactification of a
rank two symmetric space. Section~\ref{sec:solut-ODE} is devoted to
the proof of this criterion by a continuity method following the usual
steps for complex Monge-Ampère equations. The $C^0$ estimates are
obtained using essentially Wang and Zhu's method, slightly modified as
in \cite{DelKE}.  In Section~\ref{sec:constr-an-asympt}, we build an
asymptotic solution to the Ricci flat equation on a rank two symmetric
space, using as essential ingredients Stenzel's metrics and the
positive Kähler-Einstein metrics obtained in
Section~\ref{sec:KE}. This is also related to the ansatz used in
\cite{CDR16,Li17,Sze17} but is more complicated and in particular
addresses cones over singular Fano manifolds with non isolated
singularities. Finally, we detail in
Section~\ref{sec:solut-kahl-ricci} the geometry of the asymptotic
solution, and determine when the classical techniques inspired from
Tian-Yau's work apply to our setting to produce Ricci flat Kähler
metrics. The bad cases occur when the ansatz gives a metric where the collapsing towards the singular points is too quick compared to the distance in the cone: the result is a metric with holomorphic bisectional curvatures not bounded from below or from above, which is a crucial ingredient in the $C^2$ estimate for the complex Monge-Ampère equation.

\section{Setup}
\label{sec:setup}

\subsection{Symmetric spaces}

Let $G$ be a complex connected linear semisimple group.
We denote by $\left<\cdot,\cdot\right>$ the Killing form on the Lie algebra 
$\mathfrak{g}$.  
Let $\sigma$ be a complex group involution of $G$. 
Let $T_s$ be a torus in $G$ which satisfies the property that 
$\sigma(t)=t^{-1}$ for all $t\in T_s$ and maximal for this property. 
Let $T$ be a $\sigma$-stable maximal torus of $G$ containing $T_s$. 
The dimension $r$ of $T_s$ is called the \emph{rank} of the symmetric space.

Denote the root system of $(G,T)$ by $\Roots$. 
The restricted root system $\rRoots$ is the set of all non-zero characters 
of $T$ of the form $\ha-\sigma(\ha)$ for $\ha\in \Roots$. 
It forms a (possibly non-reduced) root system of rank $r$ and we let $m_{\alpha}$ 
denote the multiplicity of a restricted root $\alpha$, that is the number of 
roots $\ha\in \Roots$ such that $\alpha=\ha-\sigma(\ha)$. 
We call the Weyl group $W$ of this root system the \emph{restricted} Weyl group, etc.

We choose a positive root system $\Roots^+$ in $\Roots$ such that if 
$\alpha\in \Roots^+\setminus \Roots^{\sigma}$
then $-\sigma(\alpha)\in \Roots^+$. Then the images of elements of 
$\Roots^+\setminus \Roots^{\sigma}$ in $\rRoots$ form a positive 
restricted root system $\rRoots^+$.
We denote by $\fa$ the vector space 
$\mathfrak{t}_s\cap i\mathfrak{k}$, which is naturally identified with  
$\mathfrak{Y}(T_s)\otimes \mathbb{R}$ where $\mathfrak{Y}(T_s)$ denotes 
the group of one parameter subgroups of $T_s$.
We let $\mathfrak{a}^+$ denote the positive restricted Weyl chamber in 
$\mathfrak{a}$ defined by the choice of $\rRoots^+$.
We fix an ordering of the simple restricted roots $\alpha_1,\ldots ,\alpha_r$. 

We will use several times the symmetry of positive roots systems induced by 
a choice of simple root (see e.g. \cite[Lemma~10.2.B]{Hum78}): 
the reflection with respect to $\alpha_1$ induces 
an involution of the set $\rRoots^+\setminus \alpha_1$.

We further denote by $\varpi$ the half sum of positive restricted roots 
(counted with multiplicities) and define the numbers $A_j$ as  
the coordinates of $\varpi$ in the basis of simple roots: 
$\varpi = \sum_{j=1}^r A_j\alpha_j$.

Finally, let us introduce the Duistermaat-Heckman polynomial $P_{DH}$ of $G/H$, 
defined by 
$P_{DH}(p)=\prod_{\alpha\in\rRoots} \left<\alpha,p\right>^{m_{\alpha}}$ 
for $p\in \mathfrak{a}^*$. 

\begin{exa}
Any complex symmetric space as defined above may be recovered as the complexification 
of a compact (Riemannian) symmetric space. For example, the complexification of a 
Grassmannian leads to a complex symmetric space 
$\SL_m/S(\GL_r\times \GL_{m-r})$ for some integers $m$, $r$ with $r\leq m/2$. 
The rank of this symmetric space is $r$, 
and its positive restricted root system (of type $BC_2$) with multiplicities is depicted 
in Figure~\ref{fig:AIII} for the rank two case. 
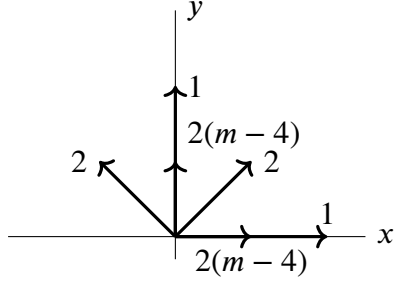
\begin{figure}
\centering
\begin{tikzpicture}
\draw[very thick] [->] (0,0) -- (1,0) node[below]{$2(m-4)$} ;
\draw[very thick] [->] (0,0) -- (2,0) node[above]{$1$} ;
\draw[very thick] [->] (0,0) -- (1,1) node[right]{$2$} ;
\draw[very thick] [->] (0,0) -- (0,1) node[above right]{$2(m-4)$} ;
\draw[very thick] [->] (0,0) -- (0,2) node[right]{$1$} ;
\draw[very thick] [->] (0,0) -- (-1,1) node[left]{$2$} ;
\draw (-2.2,0) -- (2.5,0) node[right]{$x$};
\draw (0,-.3) -- (0,3) node[right]{$y$};
\end{tikzpicture}
\caption{Restricted root system of the complexified Grassmannian}
\label{fig:AIII}
\end{figure}
\end{exa}

\begin{nota}
We will use the notations:
\[ \tilde{\alpha_1} = \alpha_1 - \frac{\left<\alpha_1,\alpha_2\right>}{\left<\alpha_2,\alpha_2\right>}\alpha_2 
\qquad \qquad 
\tilde{\alpha_2}= \alpha_2 - \frac{\left<\alpha_1,\alpha_2\right>}{\left<\alpha_1,\alpha_1\right>}\alpha_1. \]
\end{nota}

Note that 
\[ A_1=\frac{\left<\varpi,\tilde{\alpha}_1\right>}{\left<\tilde{\alpha}_1,\tilde{\alpha}_1\right>} 
\qquad \qquad 
A_2=\frac{\left<\varpi,\tilde{\alpha}_2\right>}{\left<\tilde{\alpha}_2,\tilde{\alpha}_2\right>}. \]

\subsection{The wonderful compactification}

From now on we fix an complex group involution $\sigma$.
Let $H$ be a closed subgroup of $G$ such that $\mathfrak{h}=\mathfrak{g}^{\sigma}$.
We say that a normal projective $G$-variety $X$ with given base point $x\in X$ is 
a $G$-equivariant compactification of $G/H$ if $\mathrm{Stab}_G(x)=H$ and the orbit 
of $x$ is open dense in $X$.
We will identify $G/H$ with the orbit of $x$.

Assume that $H=N_G(G^{\sigma})$. Then by \cite{DCP83} there exists a \emph{wonderful compactification} of 
$G/H$, that is, a $G$-equivariant compactification of $G/H$ which is smooth, such that 
$X\setminus G/H= \bigcup_{j=1}^r D_j$ is a simple normal crossing divisor, and  
the orbit closures of $G$ in $X$ are precisely the partial intersections $\bigcap_{j\in J} D_j$ 
for all subsets $J\subset \{1,\ldots,r\}$.
The number $r$ is the rank of the symmetric space so that in the rank two case, 
there are two codimension one orbits whose respective closures $D_1$ and $D_2$ are smooth and 
intersect transversely at $D_1\cap D_2$ which is the last orbit, of codimension two, equivariantly
isomorphic to a generalized flag manifold. 

The structure of $G$-variety on the boundary divisors $D_j$ (and more generally all orbits) 
is also known from \cite{DCP83}: there exist a parabolic subgroup $P_j$ such that 
$D_j$ is a $G$-equivariant fibration $D_j\rightarrow G/P_j$ whose fiber $X_j$ is the 
wonderful compactification of the symmetric space $L_j/N_{L_j}(L_j^{\sigma})$
where $L_j$ is a Levi subgroup of $P_j$. They are examples of horosymmetric 
varieties \cite{DelHoro}.

There is a unique $G$-stable anticanonical divisor on the wonderful compactification, 
which writes (see \emph{e.g.} \cite{Ruz12})
\[ -K_X= \sum_{j=1}^r (A_j+1)D_j. \]

The closure of the $T$-orbit of $eH$ in $G/H$ is the $T/(T\cap H)$-toric manifold $Z$
whose fan is given by the restricted Weyl chambers and their faces in 
$\mathfrak{Y}(T/T\cap H)\otimes \mathbb{R}$. Furthermore, the intersection of a 
divisor $D_j$ with $Z$ is a restricted Weyl group orbit of toric divisors in $Z$. 
The correspondence can be made explicit: consider the ray defined by the fundamental 
weight associated to $\alpha_j$ (we identify $\mathfrak{a}$ and its dual using the 
Killing form), then $D_j$ intersects $Z$ precisely along the 
toric divisor defined by this ray. 
In other words, consider the (real non-compact part of the) flat passing 
through $eH$ in $X$, equipped with the coordinates induced by the $\alpha_j$. Then 
given a sequence of points $x_k$ converging to a point $x_{\infty}\in X\setminus G/H$, 
we have $x_{\infty}\in \cap_{j\in J} D_j$, where $j\in J$ if and only if 
$\lim_{k\rightarrow \infty}\alpha_j(x_k)= \infty$.

\subsection{The Ricci flat equation}

We are interested in the Ricci flat equation $\Ric(\omega)=0$ for Kähler metrics 
on $G/H$. It is natural to impose a condition of invariance under the action of 
a maximal compact subgroup $K$ of $G$, and we furthermore assume that the Kähler 
form $\omega$ is $i\partial\bar{\partial}$-exact (note that the invariance 
condition implies the second condition provided the symmetric space is not 
Hermitian by \cite{AL92}). Then using the general setup of \cite{DelHoro}, 
one derives easily that the Ricci flat equation translates as follows. 

\begin{prop}{\cite{DelHoro}}
\label{prop:Ric-eqn}
Assume $\Psi$ is a smooth $K$-invariant strictly psh function on $G/H$ and write 
$\Psi(\exp(x)H)=\varrho(x)$ for $x\in \mathfrak{a}$. Then 
$\Ric(i\partial \bar{\partial} \Psi)=0$ if and only if $\varrho$ satisfies 
the equation
\begin{equation}\label{eq:12}
\det(d^2\varrho)\prod_{\alpha\in \rRoots^+} \left<\alpha,d\varrho\right>^{m_{\alpha}} = 
C \prod_{\alpha\in \rRoots^+} \sinh(\alpha)^{m_{\alpha}}
\end{equation}
for some constant $C>0$. 
\end{prop}

Note that it also follows from \cite{AL92} that the correspondence between $\Psi$ 
and $\varrho$ is a 1-1 correspondence between smooth $K$-invariant strictly psh 
functions on $G/H$ and smooth $W$-invariant strictly convex functions on $\fa$.
We will sometimes write $\varrho$ as $\varrho=e^{\phi}$. Then the equation writes, in terms of 
$\phi$, as 
\begin{equation}
e^{n\phi}\det(d^2\phi+(d\phi)^2)\prod_{\alpha\in \rRoots^+} \left<\alpha,d\phi\right>^{m_{\alpha}} = 
\prod_{\alpha\in \rRoots^+} \sinh(\alpha)^{m_{\alpha}},
\end{equation}
where $n$ denotes the dimension of $G/H$ and we assumed $C=1$ as we may without loss of generality.

\begin{exa}
\label{exa-Stenzel}
In the rank one case, the symmetric spaces that we defined earlier are precisely 
the complexified symmetric spaces considered by Stenzel in \cite{Ste93}. 
We may directly recover the main result of \cite{Ste93} using 
Proposition~\ref{prop:Ric-eqn}: the equation reduces to 
a one-variable ODE with separate variables of the form  
\[ \varrho''(x) (\varrho'(x))^{m_1+m_2} = C\sinh^{m_1}(x)\sinh^{m_2}(2x) \]
where $m_1$ is the multiplicity of the simple restricted root and $m_2$ the 
(possibly $0$) multiplicity of its double. Such an equation admits a unique 
even, smooth strictly convex solution, up to an additive constant, which admits a precise 
asymptotic expansion and is the Stenzel metric. We will use this metric later 
in our construction.   

In the case of $\SL_m/S(\GL_1\times \GL_{m-1})$, the complexified Grassmannian of rank one,
one has $m_1=2m-4$ and $m_2=1$, 
and there is a simple explicit solution to the above equation 
for $C=1/2$, defined by $\varrho(x)=\cosh(x)$.
\end{exa}

\begin{exa}
The first author and Paul Gauduchon provided in \cite{BG96} an explicit formula 
for the hyperKähler metric on a complexified compact Hermitian symmetric space. 
Let us see how this formula may be interpreted in our setup, for the 
complexified Grassmannian of rank two. 

We work in the coordinates $(x,y)$ defined by Figure~\ref{fig:AIII}.
Consider the function defined by $\varrho(x,y)=\cosh(x)+\cosh(y)$. 
then we compute $\partial_x\varrho=\sinh(x)$, $\partial_y\varrho=\sinh(y)$ and  
$\det(d^2\varrho)=\cosh(x)\cosh(y)$. 
Plugging this into Equation~\ref{eq:12}, we obtain the equation 
\begin{align*} 
&\cosh(x)\cosh(y)\sinh^{2(m-4)+1}(x)\sinh^{(2(m-4)+1}(y)(\sinh^2(y)-\sinh^2(x))^2=\\
&C\sinh(2x)\sinh(2y)\sinh^{2(m-4)}(x)\sinh^{2(m-4)}(y)\sinh^2(x+y)\sinh^2(y-x) 
\end{align*}
which holds for all $m$ provided $C=1/4$. Hence the function $\varrho$ 
corresponds to a Ricci flat Kähler metric, and one can check that it coincides 
with the metric of \cite{BG96}.
\end{exa}

\section{Positive Kähler-Einstein metrics on rank one horosymmetric spaces}
\label{sec:KE}

\subsection{The equation}

Let us start with a datum composed of 
a positive integer $n_1>0$, a non-negative integer $n_2\geq 0$, 
a one-variable polynomial $P$ which is positive on $]0,n_1+2n_2]$ and such that 
	$P(y)y^{-n_1-n_2}$ is an even polynomial in $y$, non-vanishing at $0$,
and a positive real number $\lambda>n_1+2n_2$ such that $P$ is non-negative on $[0,\lambda]$.
We consider the one-variable second order ordinary differential equation 
\begin{equation}
\label{eqn_genKE}
u''(x)P(u'(x))=e^{-u(x)}\sinh^{n_1}(x)\sinh^{n_2}(2x).
\end{equation}
We will use the notations
$J(x)=\sinh^{n_1}(x)\sinh^{n_2}(2x)$, and  
$P(y)=y^{n_1+n_2}(\lambda-y)^{k}\tilde{P}(y)$. Note that 
$\tilde{P}$ is positive on $[0,\lambda]$. 

We consider the weighted volume and barycenter of the segment $[0,\lambda]$ defined by 
\[ V=\int_0^{\lambda}P(y)dy \qquad \qquad
\barDH = \int_0^{\lambda} yP(y)\frac{dy}{V}. \]
We will prove in Section~\ref{sec:solut-ODE} the following statement.

\begin{thm}
\label{thm:ODE}
The numerical condition $\barDH>n_1+2n_2$ is satisfied if and only if 
there exists a smooth solution $u$ to Equation~(\ref{eqn_genKE}) 
which is strictly convex, even, and such that $u(x)-\lambda|x|=O(1)$. 
Furthermore, if it exists, the solution satisfies 
an asymptotic expansion at $+\infty$ of the form 
\[ u(x)= \lambda x +K_{0,0}+
\sum_{j,k\in \mathbb{N}, \delta\leq j\delta+2k}K_{j,k}e^{-(j\delta+2k)x} \]
with $\delta=(\lambda-n_1-2n_2)/(k+1)$, 
for some constants $K_{j,k}$.
\end{thm}

\subsection{Geometric origin of the equation: the Kähler-Einstein equation on rank one horosymmetric spaces}
\label{sec:facets}

Let $G/H$ be a rank two complex symmetric space, with corresponding involution $\sigma$. 
Choose a simple root $\ha_2$ of $G$ which gives rise to one of the simple restricted roots
$\alpha_2=\ha_2-\sigma(\ha_2)$. 
Recall from \cite{DCP83} that $\sigma$ induces a permutation 
of simple roots $\bar{\sigma}$ (caracterized by the fact that 
$\sigma(\ha)+\bar{\sigma}(\ha)$ is fixed by $\sigma$, though non-trivial in general).
Let $P$ denote the parabolic subgroup of $G$ containing $T$ such that 
$\ha_2$ and $\bar{\sigma}(\ha_2)$ are the only simple roots of $G$ which are not roots
of $P$. The Lie algebra of $P$ writes 
$\mathfrak{p}=\mathfrak{p}^r \oplus \mathfrak{l}_a \oplus \mathfrak{l}_b$
where $\mathfrak{p}^r$ is the Lie algebra of the radical of $P$, 
$\sigma$ induces a rank one (indecomposable) symmetric space on the 
semisimple factor $\mathfrak{l}_a$, and the semisimple factor 
$\mathfrak{l}_b$ is fixed by $\sigma$.

Let $L_a$ denote the simply connected semisimple group with Lie algebra 
$\mathfrak{l}_a$. There is a natural action of $P$ on the symmetric 
space $L_a/N_{L_a}(L_a^{\sigma})$, and we build from this data a rank one 
horosymmetric space $G/H_2$ under the action of $G$ by parabolic induction: 
$G/H_2$ is the quotient of $G\times L_a/N_{L_a}(L_a^{\sigma})$ by the 
diagonal action of $P$ given by $p\cdot (g,x)=(gp^{-1},p\cdot x)$. 
In order to match with the conventions of \cite{DelHoro}, 
if we let $L$ denote the Levi subgroup of $P$ containing $T$, then the involution 
of $L$ corresponding to $G/H_2$ in the definition of \cite{DelHoro} is 
the involution $\sigma_2$ defined at the Lie algebra level by $\sigma_2=\sigma$ 
on $\mathfrak{l}_a$, and $\sigma_2$ equal to the identity on the other factors 
$\mathfrak{z}(\mathfrak{l})$ and $\mathfrak{l}_b$. 

The horosymmetric space thus constructed is actually exactly the open $G$-orbit 
in the $G$-stable prime divisor $D_2$ of the wonderful compactification of $G/N_G(H)$
corresponding to the root $\alpha_2$, as one may deduce from \cite{DCP83}, or 
with some different details, from \cite{DelHoro}. 
We will call such a horosymmetric space a facet of the symmetric space $G/H$.

Let $\Roots_{Q^u}$ denote the positive roots of $G$ which have a positive coefficient 
in $\ha_2$ or $\bar{\sigma}(\ha_2)$, and let $\Roots_{a}^+$ denote the roots 
of $L_a$ (identified with roots of $G$) which are not fixed by $\sigma$. 
The restricted root system of $(L_a,\sigma|_{L_a})$ is of rank one, hence 
there are at most two possible positive restricted roots. 
We fix a simple restricted root, denoted by $\alpha_1$ (it actually corresponds 
exactly to the second simple restricted root of $G/H$). We let $n_1$ denote the 
multiplicity of $\alpha_1$, and $n_2$ denote the multiplicity of $2\alpha_1$, 
which is zero if $2\alpha_1$ is not a restricted root. 

In the situation we described above, there exists a unique \emph{colored} $\mathbb{Q}$-Fano 
compactification of $G/H_2$. This is easily seen by the classification of $\mathbb{Q}$-Fano 
compactifications of $G/H_2$ \emph{via} $\mathbb{Q}$-$G/H_2$-Gorenstein polytopes 
by Gagliardi and Hofscheier \cite{GH15_fano} and 
using the description of the colored data of horosymmetric homogeneous spaces, 
highlighted in \cite{DelHoro}. 
Note that there may exist another, non-colored $\mathbb{Q}$-Fano compactification of $G/H_2$, 
we just focus on the colored one here. 
Let $Y$ denote this colored $\mathbb{Q}$-Fano compactification of $G/H_2$. 
The moment polytope $\Delta$ for $Y$ is easily determined as 
the intersection with the positive restricted Weyl chamber of 
the line parallel to $\alpha_1$ passing through $\varpi$. 

In this setup, Theorem~\ref{thm:ODE} has the following consequence:
let $\barDH$ denote the weighted barycenter of $\Delta$ with respect to the 
Lebesgue measure, with weight the Duistermaat-Heckman polynomial of $G/H$. 

\begin{cor}
\label{cor:facets}
The $\mathbb{Q}$-Fano variety $Y$ admits a (singular) Kähler-Einstein metric 
if and only if $\barDH>n_1/2+n_2$. 
\end{cor} 

\begin{proof}
Recall that $K$ denotes a maximal compact subgroup of $G$.
Let $h$ be a smooth $K$-invariant positively curved metric on the anticanonical 
line bundle $K_{G/H_2}^{-1}$, and denote by $\omega$ its curvature form. 
The second author introduced in \cite{DelHoro} an even one-variable (in this rank one case) 
convex function $u$ associated to $h$, 
called the toric potential, and computed the curvature form $\omega$ in terms of $u$. 
It allows to write the positive Kähler-Einstein $\Ric(\omega)=\omega$ on $G/H_2$ also 
in terms of $u$.
More precisely, with the right choices of normalizing constants,  
the positive Kähler-Einstein equation writes 
\begin{equation}
u''(x)\prod_{\gamma\in \Roots_{Q^u}}\langle\gamma,2\chi-u'(x)\alpha_1\rangle 
\frac{(u'(x))^{n_1+n_2}}{\sinh^{n_1}(x)\sinh^{n_2}(2x)} = e^{-u(x)},
\end{equation}
where $\chi=\sum_{\gamma\in \Roots_{Q^u}}\frac{\gamma+\sigma_1(\gamma)}{2}
=\varpi - (n_1/2+n_2)\alpha_1$.

Define the one-variable polynomial $P$ by 
\begin{align*}
P(y)& =y^{n_1+n_2}\prod_{\beta\in \Roots_{Q^u}}\big(\langle\beta,2\chi\rangle-\langle\beta,\alpha_1\rangle y\big) \\
&= y^{n_1+n_2}\prod_{\alpha\in \rRoots^+, \alpha_1\nmid \alpha}\frac{1}{2}\big(\langle\alpha,2\chi\rangle-\langle\alpha,\alpha_1\rangle y\big)^{m_{\alpha}}
\end{align*}
where the second equality holds because $\sigma(\chi)=-\chi$. 
With these notations, the equation may be written 
\[ u''P(u')=e^{-u}J \] 
where $J(x)=\sinh(x)^{n_1}\sinh(2x)^{n_2}$.
Note that $P(y)y^{-(n_1+n_2)}$ is even thanks to the symmetry of the positive root system. 
We may also check that $P$ is positive at $n_1+2n_2$. Indeed, $n_1+2n_2$ is positive, and we have 
$2\chi+(n_1+2n_2)\alpha_1=2\varpi$, which of course satisfies that 
$\langle\gamma,2\varpi\rangle>0$ for all $\gamma\in \rRoots^+$.
In fact, $P(y)$ is, up to a multiplicative constant, equal to $P_{DH}(2\varpi-(n_1+2n_2+y)\alpha_1)$.

To see the geometric origin of the condition on the asymptotic behavior of the 
solutions $u$, we turn now to the $G$-equivariant compactification $Y$ of $G/H_1$. 
Assume that the metric $h$ extends to a locally bounded metric on $K_Y^{-1}$. 
Then, again by \cite{DelHoro}, we know that the toric 
potential $u$ has an asymptotic behavior controlled by the moment polytope $\Delta$ 
of $X$. More precisely, $\Delta$ is the translate by $\chi$ of a segment of 
the form $[0,\lambda \alpha_1]$, where 
$\lambda$ is easily derived from the description of $\Delta$: $\lambda$ is the maximum 
of all real numbers such that $\langle\alpha_2,\chi+\lambda \alpha_1\rangle\geq 0$.
The moment polytope controls the asymptotic behavior of $u$ in the sense that 
$u(x)-2\lambda|x|$ is bounded.
The value of $k$ in this setting is easily derived from the restricted root 
system by definition of $\lambda$: the number $k$ is the sum of the multiplicity 
of $\alpha_2$ and of the (possibly zero) multiplicity of $2\alpha_2$. 

Theorem~\ref{thm:ODE} thus applies to our situation, and allows to conclude.
Indeed, in the situation described, the complement $X\setminus G/H_2$ has 
codimension at least two. Furthermore, one can check that here, a locally 
bounded $K$-invariant metric on $X$ which is smooth on $G/H_2$ has full Monge-Ampère mass. 
As a consequence, finding a smooth solution $u$ to the equation, with $u(x)-|\lambda x|$ 
bounded, is equivalent to the existence of a singular Kähler-Einstein metric on $X$
(see \cite[Section~3]{BBEGZ}). 
\end{proof}

More generally, for horosymmetric (but not horospherical) spaces of rank one 
(not necessarily induced by a rank two symmetric space) the equation for 
Kähler-Einstein metrics will be of the form of the 
equation we study. Furthermore, there are variants of the Kähler-Einstein equation 
that will also be encoded by an equation of the same form. 
For example, if we consider a pair $(Y,\mu D)$ where $Y$ is a non-colored $G$-equivariant 
compactification of $G/H_2$ such that $D:=Y\setminus G/H_2$ is a divisor, $\mu>0$
and $(Y,\mu D)$ is a klt log Fano pair, then the equation for log Kähler-Einstein 
metrics in this setting is the same as above, but the real parameter $\lambda$ 
controlling the asymptotic behavior of $u$ varies with $\mu$. 

Other examples may be obtained by considering say a non-colored compactification $Y$ of 
$G/H_2$ (thus equipped with a fibration structure $\pi:Y\rightarrow G/P$) 
and considering a twisted Kähler-Einstein equation of the form 
\[ \Ric(\omega)=\omega\pm \pi^*\omega_P \]
where $\omega_P$ is some fixed $K$-invariant Kähler metric on $G/P$. 
The corresponding equation in terms of $u$ would imply a modified polynomial. 
We leave it to the interested reader to deduce the precise equation from 
\cite{DelHoro}.   

                                                                                                                                                                                                                \subsection{Existence on facets of rank two symmetric spaces}
\label{sec:checks}

In the remaining of this section, we check when the condition from Theorem~\ref{thm:ODE}
is satisfied in the examples described previously. 
Table~\ref{table:ISS2} shows the possible examples of 
indecomposable symmetric spaces of rank two (see \cite[p.532]{Hel78}). 
Note that we do not give all possible 
cases of a same given type (\emph{e.g.} the group $\SL_3\times \SL_3/\SL_3$ is also a representant 
of the group type $A_2$). 
Furthermore, we chose parameters to avoid redundancy, but some elements of the infinite 
families may also be known as representant of other families of symmetric spaces:
for example type BDI may be considered of type CI for $r=5$, of type AIII for $r=6$, 
and of type DIII for $r=8$. 
To check the condition, we reduce to three situations with parameters depending on 
the symmetric space considered. Namely we separate the possible restricted root systems 
and take as parameters the multiplicities of restricted roots as in Figure~\ref{fig:types}. 

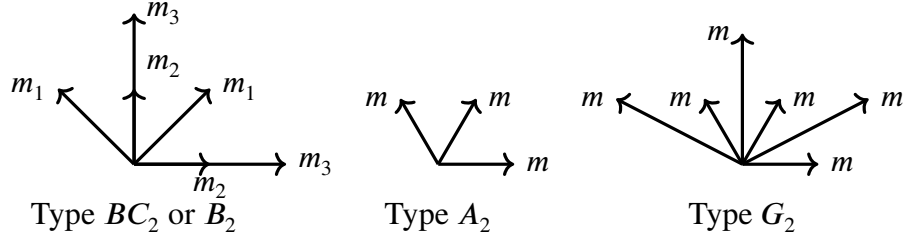
\begin{figure}
\centering
\begin{tikzpicture}
\draw[very thick] [->] (0,0) -- (1,0) node[below]{$m_2$} ;
\draw[very thick] [->] (0,0) -- (2,0) node[right]{$m_3$} ;
\draw[very thick] [->] (0,0) -- (1,1) node[right]{$m_1$} ;
\draw[very thick] [->] (0,0) -- (0,1) node[above right]{$m_2$} ;
\draw[very thick] [->] (0,0) -- (0,2) node[right]{$m_3$} ;
\draw[very thick] [->] (0,0) -- (-1,1) node[left]{$m_1$} ;
\draw (0,-.7) node{Type $BC_2$ or $B_2$};

\draw[very thick] [->] (0+4,0) -- (1+4,0) node[right]{$m$} ;
\draw[very thick] [->] (0+4,0) -- (0.5+4,0.866) node[right]{$m$} ;
\draw[very thick] [<-] (-0.5+4,0.866) node[left]{$m$} -- (0+4,0) ;
\draw (4,-.7) node{Type $A_2$};

\draw[very thick] [->] (0+8,0) -- (1+8,0) node[right]{$m$} ;
\draw[very thick] [->] (0+8,0) -- (0.5+8,0.866) node[right]{$m$} ;
\draw[very thick] [->] (0+8,0) -- (-0.5+8,0.866) node[left]{$m$} ;
\draw[very thick] [->] (0+8,0) -- (0+8,2*0.866) node[left]{$m$} ;
\draw[very thick] [->] (0+8,0) -- (1.66+8,0.866) node[right]{$m$} ;
\draw[very thick] [->] (0+8,0) -- (-1.66+8,0.866) node[left]{$m$} ;
\draw (8,-.7) node{Type $G_2$};
\end{tikzpicture}
\caption{Types of rank two root systems}
\label{fig:types}
\end{figure}

\begin{table}
\begin{equation*}
\begin{array}{ccccc}
\toprule
\mathrm{Type} & \mathrm{Parameter} & \mathrm{One~Representant} & \rRoots & \mathrm{multiplicities} \\
\midrule
\mathrm{AI} & & \SL_3/\SO_3 
& A_2 & 1 \\
A_2 & & \PGL_3\times \PGL_3/\PGL_3 
& - & 2 \\
\mathrm{AII} & & \SL_6/\Sp_6 
& - & 4 \\
\mathrm{EIV} & & E_6/F_4 & - & 8 \\
\mathrm{AIII} & r \geq 5 & \SL_r/S(\GL_2\times \GL_{r-2}) & BC_2 & (2,2r-8,1) \\
\mathrm{CII} & r\geq 5 & \Sp_{2r}/\Sp_4\times \Sp_{2r-4} & - & (4,4r-16,3) \\
\mathrm{DIII} & & \SO_{10}/\GL_5 & - & (4,4,1) \\
\mathrm{EIII} & & E_6/\SO_{10}\times\SO_2 & - & (6,8,1) \\
\mathrm{BDI} & r\geq 5 & \SO_r/S(O_2\times O_{r-2}) & B_2 & (1,r-4,0) \\
B_2& & \SO_5\times\SO_5/\SO_5 
& - & (2,2,0) \\
\mathrm{CII} & r=4 & \Sp_{8}/\Sp_4\times \Sp_{4} 
& - & (3,4,0) \\
\mathrm{G} & & G_2/\SO_4 & G_2 & 1 \\
G_2& & G_2\times G_2/G_2 & - & 2 \\ 
\bottomrule 
\end{array}
\end{equation*}
\caption{Indecomposable symmetric spaces of rank two}
\label{table:ISS2}
\end{table}

\subsubsection{Restricted root system of type $BC_2$ or $B_2$}

Note that $m_3=0$ if the root system is of type $B_2$ and else 
it is of type $BC_2$.
The possibilities for $(m_1,m_2,m_3)$ are given in Table~\ref{table:ISS2}.
We denote the simple restricted root with multiplicity $m_1$ by $\alpha$ 
and the simple restricted root with multiplicity $m_2$ by $\beta$. 
Let $\tilde{\alpha}= \alpha+\beta$ and 
$\tilde{\beta}= \alpha/2+\beta$.
We have 
\[ \varpi = (m_1+m_2/2+m_3)\alpha+(m_1+m_2+2m_3)\beta \]
and the Duistermaat-Heckman polynomial corresponding to the 
symmetric space is, in several choices 
of coordinates and up to a (different) constant factor, as follows: 
\begin{equation*}
P_{DH}(y\tilde{\alpha}+x\beta)= x^{m_2+m_3}(y^2-x^2)^{m_1}y^{m_2+m_3}
\end{equation*}
and 
\begin{equation*}
P_{DH}(w\alpha+t\tilde{\beta})=w^{m_1}t^{m_1}(t^2-(2w)^2)^{m_2+m_3}.
\end{equation*}

Depending on the choice $\alpha_1=\alpha$ or $\alpha_1=\beta$, 
there are two possible facets of $G/H$ as in the last section. 
We check when the condition of Corollary~\ref{cor:facets} is satisfied 
in each case.  
From the description in Section~\ref{sec:facets}, these conditions translate 
respectively as  
\begin{equation*}
\int_{x=0}^{m_1+m_2/2+m_3}
(x-(m_2/2+m_3))
P_{DH}((m_1+m_2/2+m_3)\tilde{\alpha})+x\beta)dx
>0
\end{equation*} 
and 
\begin{equation*}
\int_{w=0}^{m_1/2+m_2/2+m_3}
(w-m_1/2)
P_{DH}(w\alpha+(m_1+m_2+2m_3)\tilde{\beta})dx
>0.
\end{equation*} 
They may be interpreted as conditions on the weighted barycenters of the segments 
in Figure~\ref{fig:BC2pol}.

\begin{figure}
\centering
\begin{tikzpicture}
\draw[very thick] [->] (0,0) -- (1,0) node[below]{$\beta$} ;
\draw[very thick] [->] (0,0) -- (2,0);
\draw[very thick] [->] (0,0) -- (1,1);
\draw[very thick] [->] (0,0) -- (0,1) node[left]{$\tilde{\alpha}$};
\draw[very thick] [->] (0,0) -- (0,2);
\draw[very thick] [->] (0,0) -- (-1,1) node[left]{$\alpha$} ;
\draw (1/2,1/2) node{$\bullet$}; 
\draw (1/2+.3,1/2-.1) node{$\tilde{\beta}$};
\draw[thick, dashed] (0,0) -- (0,6);
\draw[thick, dashed] (0,0) -- (5,5);
\draw (2,4) node{$\bullet$}; 
\draw (2.3,4.3) node{$\varpi$};
\draw[thick] (0,4) -- (4,4);
\draw[thick] (0,6) -- (3,3);
\draw (0,4) node{$\bullet$}; 
\draw (-1.1,4.2) node{$(m_1+\frac{m_2}{2}+m_3)\tilde{\alpha}$};
\draw (3,3) node{$\bullet$}; 
\draw (4.2,2.8) node{$(m_1+m_2+2m_3)\tilde{\beta}$};
\end{tikzpicture}
\caption{Moment polytopes for type $BC_2$}
\label{fig:BC2pol}
\end{figure}
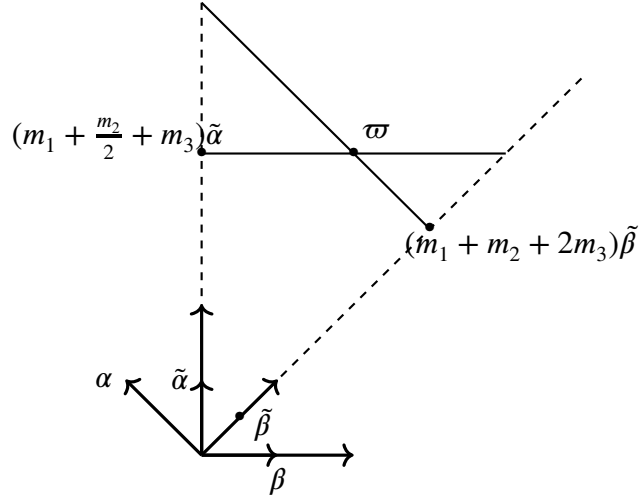

Using the changes of variables $u=(x/(m_1+m_2/2+m_3))^2$ and 
$v=(w/(m_1/2+m_2/2+m_3))^2$ and the expression of $P_{DH}$, 
we get the equivalence of the above conditions with, respectively, 
\begin{equation*}
\int_{u=0}^1 u^{(m_2+m_3)/2}(1-u)^{m_1}du 
>
\frac{m_2+2m_3}{2m_1+m_2+2m_3}\int_{u=0}^1 u^{(m_2+m_3-1)/2}(1-u)^{m_1}du
\end{equation*} 
and 
\begin{equation*}
\int_{v=0}^1 v^{m_1/2}(1-v)^{m_2+m_3}dv 
>
\frac{m_1}{m_1+m_2+2m_3}\int_{v=0}^1 v^{(m_1-1)/2}(1-v)^{m_2+m_3}dv.
\end{equation*}

These are inequalities on beta functions: recall that the beta function 
is a function of two variables defined by 
\begin{equation*}
B(\lambda,\mu)=B(\mu,\lambda)=\int_{t=0}^{1}t^{\lambda-1}(1-t)^{\mu-1}dt
\end{equation*}
Hence we want to check 
\begin{equation}
\label{cond_1}
B((m_2+m_3)/2+1,m_1+1)>\frac{m_2+2m_3}{2m_1+m_2+2m_3}B((m_2+m_3+1)/2,m_1+1)
\end{equation}
and 
\begin{equation}
\label{cond_2}
B(m_1/2+1,m_2+m_3+1)>\frac{m_1}{m_1+m_2+2m_3}B((m_1+1)/2,m_2+m_3+1).
\end{equation}

We first check these conditions by direct computation for the 
examples that do not form infinite families. In each case, we 
compute the left-hand side minus the right-hand side to check the condition.
\begin{center}
\begin{tabular}{ccc}
\toprule
$(m_1,m_2,m_3)$ & condition~(\ref{cond_1}) & condition~(\ref{cond_2}) \\
\midrule
$(2,2,0)$ & $41/1260>0$ & $1/140>0$ \\
$(3,4,0)$ & $43/7700>0$ & $83/30030>0$ \\
$(4,4,1)$ & $101/63063>0$ & $2533/1801800>0$ \\
$(6,8,1)$ & $5513/70114902>0$ & $63407/743642900>0$ \\
\bottomrule
\end{tabular}
\end{center}

For the infinite families, 
we use the expression of the beta function in terms 
of the gamma function: $B(x,y)=\Gamma(x)\Gamma(y)/\Gamma(x+y)$. 
Recall that the factorial of a positive integer is equal to the gamma function 
evaluated at the consecutive integer, and that Legendre's duplication formula 
yields the following expression, given a positive integer $p$: 
$\Gamma(p+1/2)=(2p)!\sqrt{\pi}/(p!4^p)$.

Since they are proved differently, we separate the proof for condition~(\ref{cond_1}) 
and the proof for condition~(\ref{cond_2}).

\begin{lem}
Condition~(\ref{cond_1}) is satisfied for all infinite families.
\end{lem}

\begin{proof}
This first condition is proved by direct computation. 
We provide details for the case $(m_1,m_2,m_3)=(4,4m-16,3)$ ($m\geq 4$). 
We consider the quotient of the left-hand side by the right-hand side and want to check that it is 
strictly greater than one. The quotient writes: 
\begin{align*}
& \frac{(4m-2)\Gamma(2m-6+1/2)\Gamma(2m-1)}{(4m-10)\Gamma(2m-6)\Gamma(2m-1+1/2)} \\ 
&\qquad =\frac{(4m-2)(4m-12)!(2m-2)!(2m-1)!4^{2m-1}}{(4m-10)(2m-7)!(4m-2)!(2m-6)!4^{2m-6}} \\
&\qquad = \frac{(4m-2)(4m-4)(4m-6)(4m-8)(4m-12)}{(4m-3)(4m-5)(4m-7)(4m-9)(4m-11)}
\end{align*}
It is greater than one if and only if the polynomial obtained by subtracting the 
denominator from the numerator is positive. This last polynomial is 
\begin{equation*}
768 m^4 - 5760 m^3 + 14880 m^2 - 15780 m + 5787
\end{equation*}
it has positive leading coefficient and is of degree four 
hence we may compute its roots and check 
that they are all strictly smaller than four, which means that condition~(\ref{cond_1})
is satisfied for $m\geq 4$.
\end{proof}

\begin{lem}
Condition~(\ref{cond_2}) is satisfied for all infinite families.
\end{lem}

\begin{proof}
For this condition, direct computation does not seem tractable, so we first prove 
that the quotient of the left-hand side by the right-hand side is increasing with the parameter $m$, 
for the sequence of parameters considered,  
then check that it is greater than one for the first value of the parameter.
Let us again give details on the case $(m_1,m_2,m_3)=(4,4m-16,3)$ for $m\geq 4$. 
We denote the quotient of the left-hand side by the right-hand side by $Q(m)$. 
We first compute
\begin{align*}
Q(m)&= \frac{(4m-6)\Gamma(3)\Gamma(4m-10+1/2)}{4\Gamma(2+1/2)\Gamma(4m-9)} \\
& = \frac{(4m-6)((8m-20)!)}{3\cdot 2^{8m-21}((4m-10)!)^2}
\end{align*}
Then we have 
\begin{align*}
\frac{Q(m+1)}{Q(m)}&=\frac{(4m-2)((4m-10)!)^2(8m-12)!}{2^8(4m-6)((4m-6)!)^2(8m-20)!} \\
&= \frac{(2m-1)(8m-13)(8m-15)(8m-17)(8m-19)}{(2m-3)(8m-12)(8m-14)(8m-16)(8m-18)}
\end{align*}
Again, it is greater than one if and only if the polynomial obtained by subtracting the 
denominator from the numerator is positive. This last polynomial is 
\begin{equation*}
4096 m^4 - 35584 m^3 + 113792 m^2 - 159086 m + 82167
\end{equation*}
it has positive leading coefficient and is of degree four 
hence we may compute its roots and check 
that they are all strictly smaller than four, which means that $Q(m)$ is 
increasing and thus $Q(m)\geq Q(4)$ for all $m\geq 4$. Finally, direct computation 
shows that $Q(4)=385/256>1$ hence condition~(\ref{cond_2})
is satisfied for $m\geq 4$.
\end{proof}

\subsubsection{Restricted root system of type $A_2$}

We denote the simple restricted roots by $\alpha$ and $\beta$.
There is an obvious symmetry exchanging the roles of both. 
Let $\tilde{\alpha}= \alpha+\beta/2$.
We have 
\begin{equation*}
\varpi = m(\alpha+\beta)= m\tilde{\alpha}+m\beta/2
\end{equation*}
The Duistermaat-Heckman polynomial $P_{DH}$ reads, 
up to a constant factor, as follows:  
\begin{equation*}
P_{DH}(y\tilde{\alpha}+x\beta)= 
x^m((3y/2)^2-x^2)^m.
\end{equation*}

The condition from Theorem~\ref{thm:ODE} reads as 
\begin{equation}
\label{cond_A}
\int_{x=0}^{3m/2}
(x-m/2)
x^m((3m/2)^2-x^2)^m dx
>0.
\end{equation}
It may again be interpreted as a condition on the weighted barycenter of the segment in 
Figure~\ref{fig:A2pol}.
Condition~(\ref{cond_A}) is easily checked to hold for the possible values 
of $m$ by direct computation. 

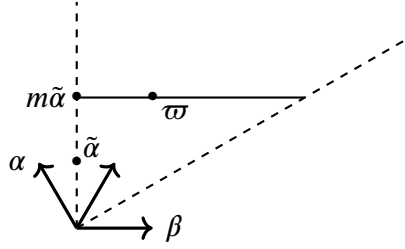
\begin{figure}
\centering
\begin{tikzpicture}
\draw[very thick] [->] (0,0) -- (1,0) node[right]{$\beta$} ;
\draw[very thick] [->] (0,0) -- (0.5,0.866)  ;
\draw[very thick] [<-] (-0.5,0.866) node[left]{$\alpha$} -- (0,0) ;
\draw (0,0.866) node{$\bullet$}; 
\draw (0.2,0.866+.2) node{$\tilde{\alpha}$};
\draw[thick, dashed] (0,0) -- (0,3);
\draw[thick, dashed] (0,0) -- (5*0.866,5/2);
\draw (1,2*0.866) node{$\bullet$}; 
\draw (1+.3,2*0.866-.2) node{$\varpi$};
\draw[thick] (0,2*0.866) -- (3,2*0.866);
\draw (0,2*0.866) node{$\bullet$}; 
\draw (-.4,2*0.866) node{$m\tilde{\alpha}$};
\end{tikzpicture}
\caption{Moment polytope for type $A_2$}
\label{fig:A2pol}
\end{figure}

\subsubsection{Restricted root system of type $G_2$}

We denote the long simple restricted root by $\alpha$ 
and the short simple restricted root by $\beta$. 
Let $\tilde{\alpha}= \alpha+3\beta/2$ and 
$\tilde{\beta}= \alpha/2+\beta$.
We have 
\[ \varpi = 3m\alpha+5m\beta \]
and the Duistermaat-Heckman polynomial $P_{DH}$ reads, in several choices 
of coordinates and up to a constant factor, as follows: 
\begin{equation*}
P_{DH}(y\tilde{\alpha}+x\beta)= 
x^my^m((3y/2)^2-x^2)^m((y/2)^2-x^2)^m
\end{equation*}
and 
\begin{equation*}
P_{DH}(w\alpha+t\tilde{\beta})=
w^mt^m((t/2)^2-w^2)^m((t/2)^2-(3w)^2)^m.
\end{equation*}

The conditions from Corollary~\ref{cor:facets} corresponding to the choices 
$\alpha_1=\alpha$ and $\alpha_1=\beta$ read as (see Figure~\ref{fig:G2pol})
\begin{figure}
\centering
\begin{tikzpicture}
\draw[very thick] [->] (0,0) -- (1,0) node[right]{$\beta$} ;
\draw[very thick] [->] (0,0) -- (0.5,0.866) ;
\draw[very thick] [->] (0,0) -- (-0.5,0.866) ;
\draw[very thick] [->] (0,0) -- (0,2*0.866);
\draw[very thick] [->] (0,0) -- (1.66,0.866) ;
\draw[very thick] [->] (0,0) -- (-1.66,0.866) node[left]{$\alpha$} ;
\draw (0,0.866) node{$\bullet$}; 
\draw (0.2,0.866+.2) node{$\tilde{\alpha}$};
\draw (1/4,0.866/2) node{$\bullet$}; 
\draw (1/4+.3,0.866/2+.1) node{$\tilde{\beta}$};
\draw[thick, dashed] (0,0) -- (0,6);
\draw[thick, dashed] (0,0) -- (7*1/2,7*0.866);
\draw (1,6*0.866) node{$\bullet$}; 
\draw (1+.2,6*0.866+.2) node{$\varpi$};
\draw (0,6*0.866) node{$\bullet$}; 
\draw (-.6,6*0.866+.2) node{$3m\tilde{\alpha}$};
\draw[thick] (0,6*0.866) -- (3,6*0.866);
\draw (5*1/2,5*.866) node{$\bullet$}; 
\draw (5*1/2+.4,5*.866-.2) node{$5m\tilde{\beta}$};
\draw[thick] (0,20/3*.866) -- (5*1/2,5*.866);
\end{tikzpicture}
\caption{Moment polytopes for type $G_2$}
\label{fig:G2pol}
\end{figure}
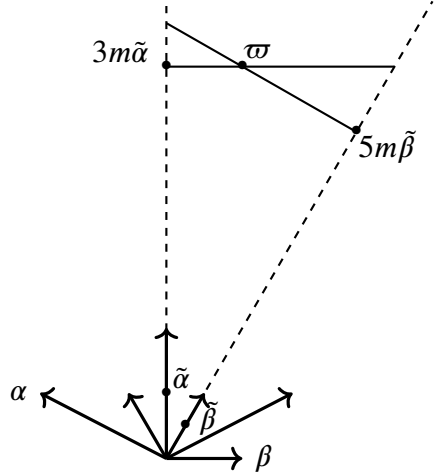
\begin{equation*}
\int_{x=0}^{3m/2}
(x-m/2)
x^m((9m/2)^2-x^2)^m((3m/2)^2-x^2)^m
dx
>0
\end{equation*}
and 
\begin{equation*}
\int_{w=0}^{5m/6}
(w-m/2)
w^m((5m/2)^2-w^2)^m((5m/2)^2-(3w)^2)^m
dx
>0.
\end{equation*}

Direct computation shows that the first condition holds 
for $m=1$ (the integral is equal to $12879/1792$) 
and $m=2$ (the integral is then equal to $192283227/308$). 
The second condition, on the other hand, is not satisfied:
the integral is equal to $-171875/435456$ if $m=1$, and 
to $-79443359375/6062364$ if $m=2$.

\section{Solution to the ODE by the continuity method}
\label{sec:solut-ODE}

\subsection{The continuity method}

To prove the existence of a solution, we consider the family of equations 
\begin{equation}
\label{eqn_u_t}
u_t''(x)P(u_t'(x))=e^{-(tu_t(x)+(1-t)u_{\mathrm{ref}}(x))}J(x)
\end{equation} 
indexed by $t\in [0,1]$. 
Here, $u_{\mathrm{ref}}$ is the smooth, even, strictly convex function 
on $\mathbb{R}$ defined by 
\begin{equation}
u_{\mathrm{ref}}(x)=\ln(e^{\lambda x}+e^{-\lambda x})+C
\end{equation}
where $C$ is the constant determined by the condition 
$\int_0^{\infty}e^{-u_{\mathrm{ref}}}J = \int_0^{\lambda}P$.

Consider the set $I\subset [0,1]$ of all $t$ such that there exists an even, $C^2$ 
solution $u_t$ to Equation~(\ref{eqn_u_t}) with $u_t'(\mathbb{R})=]-\lambda,\lambda[$.
We will show:

\begin{prop}
\label{prop:cont}
The set $I$ is equal to 
$ [0,1]\cap [0,(\lambda-n_1-2n_2)/(\lambda-\barDH)[ $
\end{prop}

\subsection{Asymptotic expansion of the solutions}

Before proving Proposition~\ref{prop:cont}, let us prove the second half of Theorem~\ref{thm:ODE}, 
that is, the asymptotic expansion of solutions. 
Recall the notation from Theorem~\ref{thm:ODE}
\[ \delta=\frac{\lambda-n_1-2n_2}{k+1}. \]

\begin{prop}
\label{prop:cont-exp}
Let $u_t$ be a $C^2$, even solution to Equation~(\ref{eqn_u_t}) 
such that $u_t'(\mathbb{R})=]-\lambda,\lambda[$. Then $u_t$ is smooth, 
strictly convex, and admits an arbitrarily precise expansion at infinity:
for any integer $j_m$, there are constants $K_{t,j,k}$ 
such that 
\[ u_t(x)= \lambda x +K_{t,0,0}+
\sum_{\delta\leq j\delta+2k\leq j_m\delta}K_{t,j,k}e^{-(j\delta+2k)x}
+o(e^{-j_m\delta x}) \]
\end{prop}

\begin{proof}
Assume that $u_t$ is a $C^2$, even solution to Equation~(\ref{eqn_u_t}) 
such that $u_t'(\mathbb{R})=]-\lambda,\lambda[$. The parity of $u_t$, together 
with the order of vanishing of $J$ at $0$, imply that $u_t'$ vanishes to order 
exactly one at $0$, and that $u_t''$ is positive everywhere. It shows that $u_t$ 
is strictly convex, and using the equation inductively, that $u_t$ is smooth. 

By convexity and the assumption  $u_t'(\mathbb{R})=]-\lambda,\lambda[$, we deduce 
that $u_t(x)-\lambda x$ admits a finite limit $K_{t,0,0}$ at infinity, which 
provides the two initial terms of the expansion formula, and the full expansion 
formula for $j_m=0$.

We proceed now by induction and assume that the expansion formula is proved 
for a given $j_m$. We will prove an expansion formula for 
$j_m+1$. 

Consider the function $F$ defined for $w\geq 0$ by
\[ F(w)=\left( \frac{(k+1)!(Q(\lambda)-Q(\lambda-w))}{(-1)^kP^{(k)}(\lambda)} \right)^{1/(k+1)}. \]
Note that the assumptions on $P$ imply that $(-1)^kP^{(k)}(\lambda)>0$. 
The function $F$ admits an expansion to any arbitrary order $N$ 
\[ F(w)= \sum_{1\leq n\leq N}A_nw^n +o(w^N) \] 
where $A_1=1$, $A_2=-\frac{P^{(k+1)}(\lambda)}{(k+1)(k+2)P^{(k)}(\lambda)}$, etc. 
It is invertible near $0$ and its inverse function $G$ satisfies an expansion 
\[ G(s)= \sum_{1\leq n\leq N}B_ns^n +o(s^N) \]
to any order $N$, with $B_1=1$, 
$B_2=\frac{P^{(k+1)}(\lambda)}{(k+1)(k+2)P^{(k)}(\lambda)}$, etc.

Using the definition of $F$ and the equation, we have  
\[ F(\lambda-u_0'(x)) =
\left(\frac{(k+1)!}{(-1)^kP^{(k)}(\lambda)} 
\int_x^{\infty}e^{-(tu_t+(1-t)u_{\mathrm{ref}})}J\right)^{1/(k+1)} \]

The function $u_{\mathrm{ref}}$ obviously admits an expansion as in the 
statement at any order, hence we have an expansion 
\[ tu_t(x)+(1-t)u_{\mathrm{ref}}= \lambda x +K_{0,0}+
\sum_{\delta\leq j\delta+2k\leq j_m\delta}K_{j,k}e^{-(j\delta+2k)x}
+o(e^{-j_m\delta x}), \]
for some constants $K_{j,k}$.
We may thus write an expansion formula
\begin{multline*}
\frac{(k+1)!}{(-1)^kP^{(k)}(\lambda)}e^{-(tu_t(x)+(1-t)u_{\mathrm{ref}})}J \\
= K_{0,0}'e^{(n_1+2n_2-\lambda)x} ( 1+ 
\sum_{\delta\leq j\delta+2k\leq j_m\delta}K_{j,k}'e^{-(j\delta+2k)x}
+o(e^{-j_m \delta x}) )
\end{multline*}
for some constants $K_{j,k}'$. Up to replacing the constants $K_{j,k}'$ 
by others constants $K_{j,k}''$, the expansion is still valid for the integral 
from $x$ to infinity. 
Taking the power $1/(k+1)$ we obtain the expansion 
\begin{align*} 
F(\lambda-u_0'(x)) & = K_{0,0}^{(3)}e^{-\delta x} \left( 1 + 
\sum_{\delta\leq j\delta+2k\leq j_m\delta}K_{j,k}^{(3)}e^{-(j\delta+2k)x}
+o(e^{-j_m \delta x}) \right) 
\\
& 
=  K_{0,0}^{(3)}e^{-\delta x} + 
\sum_{\delta\leq j\delta+2k\leq j_m\delta}K_{j,k}^{(4)}e^{-((j+1)\delta+2k)x}
+o(e^{-(j_m+1)\delta x}) 
\end{align*}

We finally apply $G$ to the expansion of $F(\lambda-u_0')$ to deduce the 
corresponding expansion of $u_0'$: 
\[ u_0'(x)=\lambda+ K_{0,0}^{(5)}e^{-\delta x} + 
\sum_{\delta\leq j\delta+2k\leq j_m\delta}K_{j,k}^{(5)}e^{-((j+1)\delta+2k)x}
+o(e^{-(j_m+1)\delta x}). \]
This expansion integrates to provides the expansion of $u_0$ at the 
order $j_m+1$. 
\end{proof}

\subsection{Initial solution}

We now proceed to the proof of Proposition~\ref{prop:cont}, and first verify $0\in I$. 
The equation for $t=0$ is an ordinary differential equation with separate variables. 
Let $Q$ denote a fixed primitive of $P$. 
It is strictly increasing on $[0,\lambda]$. Let $Q^{-1}$ denote its inverse 
function, that is such that $Q^{-1}(Q(y))=y$ 
for $0\leq y \leq \lambda$. 
Let $u_0:\mathbb{R}\rightarrow\mathbb{R}$ denote the even function defined 
for $x$ non-negative by 
\[ u_0(x)=\int_0^x Q^{-1}\left(Q(0)+\int_0^s e^{-u_{\mathrm{ref}}}J \right)ds. \]
It is easily checked to be a $C^2$ solution 
to Equation~(\ref{eqn_u_t}) at $t=0$, 
with $u_0'(\mathbb{R})=]-\lambda,\lambda[$, 
then is ultimately smooth and strictly convex in view of Proposition~\ref{prop:cont-exp}.  

\subsection{Upper bound on the time of existence of a solution}

\begin{prop}
Assume that there exists an even and $C^2$ solution $u_t$ 
to Equation~(\ref{eqn_u_t}) at time $t$ 
with $u_t'(\mathbb{R})=]-\lambda,\lambda[$, then  
\[ t< \frac{\lambda-n_1-2n_2}{\lambda-\barDH}. \]
\end{prop}

\begin{proof}
Assume there exists a solution as in the statement. Then it is in particular 
strictly convex. 
It is part of our assumptions 
that $\lambda>n_1+2n_2$, hence $e^{-(tu_t+(1-t)u_{\mathrm{ref}})}J$ 
converges to zero at infinity. 
As a consequence, the integral from zero to infinity of the derivative 
$(e^{-(tu_t+(1-t)u_{\mathrm{ref}})}J)'$ 
vanishes. 
On $]0,\infty[$, this derivative is equal to 
\[
e^{-(tu_t+(1-t)u_{\mathrm{ref}})(x)}J(x)\left(n_1\coth(x)+
2n_2\coth(2x)-(tu_t+(1-t)u_{\mathrm{ref}})'(x)\right) \]

Using Equation~(\ref{eqn_u_t}), then the change of variables $y=u_t'(x)$, 
we get 
\begin{align*}
\int_0^{\infty}u_t'(x)e^{-(tu_t+(1-t)u_{\mathrm{ref}})(x)}J(x)dx & =
\int_0^{\infty}u_t'(x)P(u_t'(x))u_t''(x)dx \\
& = \int_0^{\lambda}yP(y)dy \\
& = V \barDH
\end{align*}

On the other hand, we have $\cosh>\sinh$ and $u_{\mathrm{ref}}'<\lambda$ 
hence the vanishing of the integral of $(e^{-(tu_t+(1-t)u_{\mathrm{ref}})}J)'$ 
yields the inequality 
\[ n_1+2n_2-t\barDH-(1-t)\lambda<0. \]
We have thus obtained the desired necessary condition.
\end{proof}

\subsection{Openness} 
Just as in choosing the continuity method to solve the equation, we proceed here in analogy 
with the case of Kähler-Einstein metrics on compact manifolds. This is even more justified 
as in the case that interests us the most, we are working on a singular complex variety. 
The openness follows from the usual method in the Kähler-Einstein continuity method, except that since our manifold is singular, we must use weighted spaces instead of the standard functional spaces.
Denote by $C^{k,\ev}$ the space of even $C^k$ functions on $\Bbb{R}$. To solve the equation, we use weighted spaces
\begin{equation}
  C^{k,\ev}_\eta=\cosh(x)^\eta C^{k,\ev} .
\end{equation}
We drop the suffix $\ev$ if we consider the same space only on an interval $(A,\infty)$ with $A>0$.

We rewrite Equation (\ref{eqn_u_t}) as
\begin{equation}
  \label{eq:3}
  \ln \big( u_t''P(u'_t) \big) + t u_t = -(1-t)u_{\mathrm{ref}} + \ln J.
\end{equation}
The linearization of the LHS is
\begin{equation}
  L_tv = \Delta_tv + t v, \quad \Delta_tv = \frac{v''}{u_t''} + \frac{P'(u'_t)}{P(u'_t)} v' .
\end{equation}

\begin{lem}\label{lem:vp-D}
  If $u_t$ is a solution of \eqref{eqn_u_t}, and we denote $\omega_t=u''_t P(u'_t)$ the volume form of the corresponding metric, then one has the estimate
  \[ \int (\Delta_tv)^2 \omega_t \geq t \int \frac{(v')^2}{u''_t} \omega_t, \]
and the inequality is strict if $v'\neq0$.
\end{lem}
\begin{proof}
  This is the usual estimate for the first nonzero eigenvalue of the Laplacian in the continuity method: since $-\ln J$ and $u_{\mathrm{ref}}$ are convex, the equation on $u_t$ implies
  \begin{equation}
 \rho_t := -(\ln \omega_t)'' > t u_t''.\label{eq:5}
\end{equation}
(This is a weaker version of  $\Ric> t$ which writes $\rho_t+(\ln J)''>t u_t''$).

To prove the estimate, we might check that the usual Weitzenböck formula applies (we are on a singular manifold), but in our case it is easy to reprove it directly: by integration by parts, writing $\Delta_tv=\frac{(P(u_t')v')'}{P(u_t')u_t''}$, one obtains
  \begin{align*}
    \int \frac{\rho_t(v')^2}{(u''_t)^2} \omega_t
    & = \int -(\ln \omega_t)'' \frac{(v' P(u_t'))^2}{\omega_t} \\
    & = \int 2 \frac{\omega'_t}{u''_t} v' \Delta_tv - \frac{(\omega'_t)^2}{(u''_t)^3P(u'_t)} (v')^2 \\
    & \leq \int (\Delta_tv)^2 \omega_t
  \end{align*}
  and the result follows from \eqref{eq:5}, as does the strict inequality.
\end{proof}

From Proposition~\ref{prop:cont-exp} we have 
$u_t'(x)=\lambda-K_2\delta e^{-\delta x}+O(e^{-(\delta+\epsilon)x})$ and 
$u_t''(x)=K_2\delta^2e^{-\delta x}+O(e^{-(\delta+\epsilon)x})$. Therefore the leading terms of $\Delta_t$ are given by
\[ \Delta_tv \sim -\frac{e^{\delta x}}{K_2\delta^2} ( v'' - k\delta v'). \]

So it is natural to study the operator $L_t=\Delta_t+t:C^{2,\ev}_\eta\rightarrow C^{0,\ev}_{\delta+\eta}$. Observe that there is an asymptotic solution converging to a constant at infinity: if near $\infty$
\begin{equation}
v_0(x) = 1 + \frac{tK_2}{\delta} e^{-\delta x}
\end{equation}
then
\begin{equation}
L_tv_0=O(e^{-\eta_0 x})\label{eq:6}
\end{equation}
for small $\eta_0$.
We extend $v_0$ as an even function on $\Bbb{R}$.

\begin{lem}\label{lem:L-iso}
  If $-\delta-\eta_0\leq\eta<0$ and $t>0$ then $L_t:\Bbb{R}v_0 \oplus C^{2,\ev}_\eta\rightarrow C^{0,\ev}_{\delta+\eta}$ is an isomorphism.
\end{lem}
\begin{proof}
  Weighted analysis (see for example \cite{LocMcO}) says immediately that \[L_t:C^{2,\ev}_\eta\rightarrow C^{0,\ev}_{\delta+\eta}\] is Fredholm as soon as $\eta\neq0, k\delta$, which are the critical weights giving the possible orders of growth of elements of the kernel of $L_t$. Moreover $L_t$ is selfadjoint with respect to the volume form $\omega_t\sim \mathrm{cst.} e^{-(k+1)\delta x}$.

  The $L^2$ space corresponds to the weight $\frac12 (k+1)\delta$, and the same weighted analysis implies that $\Delta_t$ has discrete spectrum; from lemma \ref{lem:vp-D}, the first nonzero eigenvalue of $\Delta_t$ is greater than $t$, and therefore $\ker_{L^2}L_t=0$. This implies that the kernel of $L_t$ in $C^{2,\ev}_\eta$ vanishes for $\eta<\frac12 (k+1)\delta$ and therefore for $\eta<k\delta$ since no kernel can appear between critical weights.
  
  From selfadjointness, the cokernel of $L_t$ for the weight $\eta$ identifies to the kernel of $L_t$ for the weight $-\eta+k\delta$, so we get surjectivity provided that $\eta>0$. When the weight $\eta$ crosses the critical weight $0$, the index changes by $1$, so we get for $\eta<0$ an index equal to $-1$. If we add the factor $\Bbb{R}v_0$ at the source, we therefore obtain a Fredholm operator of index 0; it is an isomorphism since $L_t$ is injective for weights smaller then $k\delta$. The restriction $\eta\geq-\delta-\eta_0$ comes from \eqref{eq:6}, one may obtain the isomorphism for smaller $\eta$ provided that $v_0$ is replaced by an asymptotic solution to order $\delta+\eta$.
\end{proof}

\begin{proof}[Proof of openness]
  For $t>0$ the operator $L_t$ is an isomorphism between the spaces specified in Lemma \ref{lem:L-iso}, which is exactly what we need to apply the implicit function theorem to equation \eqref{eq:3}. For $t=0$, as is well-known, one recovers the same result by applying the implicit function theorem to the operator $\ln \big( u_t''P_1(u'_t) \big) + t u_t + \int u_t\omega_0$.
\end{proof}

\subsection{$C^0$ estimates}

We turn now to \emph{a priori} estimates on the solutions to Equation~(\ref{eqn_u_t}).
We begin with $C^0$ estimates with respect to the function $u_0$, 
which are the estimates where the condition appears. 
Our goal in this section is thus to prove the existence, 
on any closed interval $[t_0,t_1]\subset [0,1]\cap ]0,(\lambda-n_1-2n_2)/(\lambda-\barDH)[$ 
of a constant $C$ such that 
$|u_t-u_0|\leq C$ for any smooth, even, strictly convex solution $u_t$ of 
Equation~(\ref{eqn_u_t}) with $|u_t(x)-\lambda|x||=O(1)$, at time $t\in [t_0,t_1]$. 

In the following, $u_t$ denotes a smooth, even, strictly convex solution of 
Equation~(\ref{eqn_u_t}) at time $t$ with $u_t(x)-\lambda|x|=O(1)$. 
Set $j=-\ln J$ on $]0,\infty[$ and 
$ \nu_t:=tu_t+(1-t)u_{\mathrm{ref}}+j $.
Note that, on $]0,\infty[$, $e^{-\nu_t}$ is the right-hand side of Equation~(\ref{eqn_u_t}). 
In particular, its integral is fixed:
\begin{equation}
\int_0^{\infty} e^{-\nu_t} dx= V.
\end{equation}

The function $\nu_t$ is smooth and strictly 
convex and satisfies 
$\lim_{x\rightarrow 0}\nu_t(x)=
\lim_{x\rightarrow +\infty}\nu_t(x)=+\infty$.
As a consequence, $\nu_t$ admits a unique minimum and we introduce 
the notations $m_t$ and $x_t$ defined by  
$m_t=\min_{]0,\infty[}\nu_t=\nu_t(x_t)$.

\subsubsection{Reducing to estimates on $m_t$, $x_t$, and linear growth}

\begin{lem}
\label{lem_reduction}
Assume there exists positive constants $t_0$, $C_m$, $C_x$, $\ell_1$ and $\ell_0$ 
such that $t\geq t_0$, $|m_t|<C_m$, $\nu_t(x)\geq \ell_1|x-x_t|-\ell_0$, and $x_t<C_x$.
Then $\sup |u_t-u_0|\geq C$ 
for some constant $C$ independent of $t\geq t_0$.
\end{lem}

\begin{proof}
Denote by $v_t$ resp. $v_0$ the Legendre transforms 
of $u_t$ and $u_0$. They are even and bounded strictly 
convex functions defined on $[-\lambda,\lambda]$, smooth on $]-\lambda,\lambda[$
and continuous on $[-\lambda,\lambda]$. 
It is standard that $\sup_{\mathbb{R}}|u_t-u_0| = \sup_{[0,\lambda]} |v_t-v_0|$.
To prove the statement, it is thus enough 
to bound $|v_t|$ on $[0,\lambda]$.

Let $\hat{v}_t:=\int_0^{\lambda}v_tdp/(\lambda)$ denote the mean value of $v_t$. 
By Morrey's inequality, then by the Poincaré-Wirtinger inequality, we have (for some 
constant $C$ independent of $t$ which may change from line to line)
\begin{align*}
|v_t-\hat{v}_t|_{C^{0,1/2}} &\leq C \left(|v_t-\hat{v}_t|_{L^{2}}+|v_t'|_{L^{2}}\right) \\
& \leq C |v_k'|_{L^{2}}
\end{align*}
Choose $p$, $q>1$ such that $1/p+1/q=1$ and $P^{-q/p}$ is integrable on $[0,\lambda]$. 
Then by Holder's inequality, we can write 
\begin{align*}
\int_0^{\lambda} |v_t'|^2 & = \int_0^{\lambda} (|v_t'|^2P^{1/p}) (P^{-1/p}) \\
&\leq \left(\int_0^{\lambda} |v_t'|^{2p}P\right)^{1/p}\left(\int_0^{\lambda} P^{-q/p}\right)^{1/q} \\
& \leq C \left(\int_0^{\lambda} |v_t'|^{2p}P\right)^{1/p}
\end{align*}
By the change of variables $x=v_k'$, we have 
\begin{align*}
\int_0^{\lambda} |v_t'|^{2p}P & = \int_0^{\infty} |x|^{2p} P(u_t'(x))u_t''(x) dx \\
& = \int_0^{\infty} |x|^{2p} e^{-\nu_t(x)}dx \\
\intertext{by Equation~(\ref{eqn_u_t}). By the linear growth estimate, this is}
& \leq \int_0^{\infty} |x|^{2p} e^{-\ell_1|x|+\ell_0+\ell_1C_x} dx = C
\end{align*}
We thus have $|v_t-\hat{v}_t|_{C^{0,1/2}} \leq C$.
As a consequence, 
\begin{equation*}
\sup_{y_1,y_2\in [0,\lambda]} |v_t(y_1)-v_t(y_2)| \leq C. 
\end{equation*}
Hence to conclude it suffices to bound $v_t$ at some point.
By definition of Legendre transform, $u_t(0)=-v_t(u_t'(0))=-v_t(0)$. 
Since $|u_t'|\leq \lambda$, we have 
\begin{equation*}
|u_t(0)|\leq |u_t(x_t)|+\lambda C_x.
\end{equation*}
Note that there exists a constant $s_1>0$, independent of $t$, such that 
$x_t\geq s_1$.
Indeed, the minimum $x_t$ is the point where $\nu_t'=0$.
Since $j$ tends to infinity near $0$, its derivative is unbounded, 
whereas $u_t'$ and $u_{\mathrm{ref}}'$ are $\leq \lambda$. 

By definition of $m_t$ and $x_t$, we can conclude: 
\begin{align*}
|u_t(x_t)|&=\frac{1}{t}\left| m_t-(1-t)u_{\mathrm{ref}}(x_t)-j(x_t)\right| \\
& \leq \frac{1}{t_0}\left( C_m+\sup_{[s_1,C_x]}u_{\mathrm{ref}}+\sup_{[s_1,C_x]}j \right) \\
& \leq C
\end{align*} 
\end{proof}

\subsubsection{Estimates on $|m_t|$ and linear growth}

Define $0<\delta=\delta(t)<y=y(t)$ by $[y-\delta,y+\delta]=\nu_t^{-1}([m_t,m_t+1])$.
Note that there exists an $s_2>0$ independent of $t$ such that $y-\delta\geq s_2$.
Indeed, for $0<x<1$, consider the expression 
\[ \nu_t(x) = \nu_t(1) + \int_1^x\nu_t'(z)dz \geq m_t+\int_1^x j'(z)dz \]
Since $j'$ is negative and $\int_1^0 j'(z)dz=\infty$,   
we may find $s>0$ such that $\int_1^x j'(z)dz\geq 1$ for all 
$0<x<s$, hence $\nu_t(x)\geq m_t+1$ for $x<s$. 

On $[s_2,\infty[$, the derivatives of $\nu_t$ admit a uniform bound independent
of $t$, so we may also find a $\delta_0>0$ independent of $t$ with 
$[x_t-\delta_0,x_t+\delta_0]\subset [y-\delta,y+\delta]$.

We will use estimates on $\delta$ to derive estimates on $|m_t|$ and linear growth.

\begin{lem}
\label{lem_estimate_delta}
Assume $t\geq t_0>0$ then $\delta \leq \sqrt{t_0^{-1}e^{m_t+1}\sup_{[0,\lambda]}P}$.
\end{lem}

\begin{proof}
Consider the function $f$ defined by 
\begin{equation*}
f(x)=\nu_t(x)-t_0e^{-m_t-1}(\sup_{[0,\lambda]}P)^{-1}\left((x-y)^2-\delta^2\right)-m_t-1.
\end{equation*}
We claim that $f$ is convex on $[y-\delta,y+\delta]$. Indeed:
\begin{align*}
f''(x)&=tu_t''(x)+(1-t)u_{\mathrm{ref}}''(x)+j''(x)-t_0e^{-m_t-1}(\sup_{[0,\lambda]}P)^{-1} \\
&\geq t_0u_t''(x)-t_0e^{-m_t-1}(\sup_{[0,\lambda]}P)^{-1} \\
&\geq t_0e^{-\nu_t(x)}(P(u_t'(x)))^{-1}-t_0e^{-m_t-1}(\sup_{[0,\lambda]}P)^{-1} \\
&\geq 0
\end{align*}
where the last inequality holds by definition of $y$ and $\delta$.
By construction, $f(y-\delta)=f(y+\delta)=0$, hence by convexity, $f(y)\leq 0$.
This translates as the second inequality in 
\begin{equation*}
m_t\leq\nu_t(y)\leq -t_0e^{-m_t-1}(\sup_{[0,\lambda]}P)^{-1}\delta^2+m_t+1
\end{equation*}
and concludes the proof.
\end{proof}

\begin{prop}
\label{prop_m_t}
There exists positive constants $C_m$, $\ell_1$, $\ell_0$ such 
that $|m_t|\leq C_m$ and $\nu_t(x)\geq \ell_1|x-x_k|-\ell_0$. 
\end{prop}

\begin{proof}
We use Donaldson's coarea formula \cite{Don08} to express $V$: 
\begin{equation*}
V= \int_0^{\infty} e^{-\nu_t(x)}dx = e^{-m_t} \int_0^{\infty}e^{-s}\mathrm{Vol}(\{\nu_t\leq m_t+s\}) ds
\end{equation*}
We first obtain both upper and lower bounds on $\mathrm{Vol}(\{\nu_t\leq m_t+s\})$.
On one hand, for $s\geq 1$, the set $\{\nu_t\leq m_t+s\}$ contains 
$\{\nu_t\leq m_t+1\}=[y-\delta,y+\delta]$ 
hence also $[x_t-\delta_0,x_t+\delta_0]$, so that, for $s\geq 1$,
\begin{equation*}
\mathrm{Vol}(\{\nu_t\leq m_t+s\})\geq 2\delta_0.
\end{equation*}
On the other hand, by convexity, the set $\{\nu_t\leq m_t+s\}$ is included in 
the $s$-dilation of $[y-\delta,y+\delta]$ with center $x_t$. 
As a consequence, 
\begin{equation*}
\mathrm{Vol}(\{\nu_t\leq m_t+s\})\leq 2s\delta \leq 2s\sqrt{t_0^{-1}e^{m_t+1}\sup_{[0,\lambda]}P}
\end{equation*} 
where the last inequality follows from Lemma~\ref{lem_estimate_delta}.

From this we deduce upper and lower bounds on $V$: on one hand, 
\begin{align*}
V& \geq e^{-m_t} \int_1^{\infty}e^{-s}\mathrm{Vol}(\{\nu_t\leq m_t+s\}) ds\\
&\geq 2\delta_0 e^{-m_t} \int_1^{\infty}e^{-s}ds = 2\delta_0e^{-m_t-1}
\end{align*}
and on the other hand 
\begin{equation*}
V\leq 2e^{-m_t}\sqrt{t_0^{-1}e^{m_t+1}\sup_{[0,\lambda]}P}\int_0^{\infty}se^{-s}ds
=2\sqrt{t_0^{-1}e\sup_{[0,\lambda]}P}e^{-m_t/2}.
\end{equation*}

We easily translate this into a bound $|m_t|\leq C_m$. 

Going back to Lemma~\ref{lem_estimate_delta}, we now have a constant 
$\delta_m$ independent of $t$ such that $\delta\leq \delta_m$.
As a consequence, we have $\nu_t(x_t\pm 2\delta_m)\geq m_t+1$ and, by convexity,
$\nu_t(x)\geq |x-x_t|/(2\delta_m)+m_t$ 
outside of the interval $[x_t-2\delta_m,x_t+2\delta_m]$. 
The conclusion thus follows: 
\begin{equation*}
\nu_t(x)\geq |x-x_t|/2\delta_m+m_t-1 \geq |x-x_t|/2\delta_m-C_m-1
\end{equation*}
everywhere.
\end{proof}

\subsubsection{End of proof of $C^0$ estimates}

We conclude the proof by contradiction. By openness at $0$ it means that the 
$C^0$ estimates fail on some interval 
$[t_0,t']\subset [0,1]\cap [0,(\lambda-n_1-2n_2)/(\lambda-\barDH)[$
 with $t_0>0$.  
Then we may find a sequence $(t_k)_{k\in\mathbb{N}^*}$ of elements of $[t_0,t']$ 
such that $t_k\rightarrow t_{\infty}$ and 
\begin{equation*}
\lim_{k\rightarrow \infty} \sup_{\mathbb{R}}|u_{t_k}-u_0| =\infty.
\end{equation*} 
By Lemma~\ref{lem_reduction} and Proposition~\ref{prop_m_t}, 
we then have $\lim_{k\rightarrow \infty} x_{t_k}=\infty$ 
up to passing to a subsequence. 

In view of the properties of $\nu_t$, it is immediate that   
\begin{equation*}
\int_0^{\infty} \nu_{t_k}'e^{-\nu_{t_k}}dx =0
\end{equation*}
This vanishing integral may be rewritten as  
\begin{equation}
\label{eq:vanishing}
t_k(\int_0^{\infty}u_{t_k}'e^{-\nu_{t_k}}+\int_0^{\infty}j'e^{-\nu_{t_k}})=
(t_k-1)(\int_0^{\infty}u_{\mathrm{ref}}'e^{-\nu_{t_k}}+\int_0^{\infty}j'e^{-\nu_{t_k}}).
\end{equation}

\begin{lem}
\label{lem_limits}
The limit of equality \eqref{eq:vanishing} as $k\rightarrow \infty$ gives 
\begin{equation*}
t_{\infty}(\barDH-(m_{\alpha_1}+2m_{2\alpha_1}))=(t_{\infty}-1)(\lambda-(m_{\alpha_1}+2m_{2\alpha_1})).
\end{equation*}
\end{lem}

Before proving the lemma, we show that it allows to conclude.
Indeed, Lemma~\ref{lem_limits} implies $t_{\infty}=(\lambda-n_1-2n_2)/(\lambda-\barDH)$, which is 
a contradiction with $t_{\infty}\leq t'<(\lambda-n_1-2n_2)/(\lambda-\barDH)$. 

\begin{proof}
By Equation~(\ref{eqn_u_t}), Legendre transform 
and the definition of $\barDH$, we have 
$\int_0^{\infty}u_{t_k}'e^{-\nu_{t_k}}= V\cdot \barDH$ for all $k$.

Let us abbreviate indices $t_k$ by $k$ in the rest of the proof.
Let $\epsilon>0$.
Recall that $\nu_k(x)\geq  \ell_1|x-x_k|-\ell_0$.
We may thus fix a $\delta>0$ independent of $k$ such that 
\begin{equation*}
\int_{]0,\infty[\setminus [x_k-\delta,x_k+\delta]}e^{-\nu_k(x)}dx 
\leq \int_{\mathbb{R}\setminus [x_k-\delta,x_k+\delta]}e^{-( \ell_1|x-x_k|-\ell_0)}dx 
\leq \epsilon,
\end{equation*}
and 
\begin{equation*}
e^{-\nu_k(x_k\pm\delta)}<\epsilon.
\end{equation*}

Fix some $s>0$. Since $x_k\rightarrow \infty$, there exists $k_0$ 
such that for any $k\geq k_0$,
$-\epsilon<u_{\mathrm{ref}}'-\lambda<0$ and  
$-\epsilon<j'+m_{\alpha_1}+2m_{2\alpha_1}<0$ on $[x_k-\delta,x_k+\delta]$.

Then we can write, for $k\geq k_0$, 
\begin{align*}
\left|\int_0^{\infty}u_{\mathrm{ref}}'e^{-\nu_k} -\lambda V\right|
& \leq \left|\int_{]0,\infty[\setminus [x_k-\delta,x_k+\delta]}u_{\mathrm{ref}}'e^{-\nu_k}\right|+
\left|\int_{[x_k-\delta,x_k+\delta]}(u_{\mathrm{ref}}'-\lambda)e^{-\nu_k}\right| \\
& \qquad +\left|\lambda\int_{]0,\infty[\setminus [x_k-\delta,x_k+\delta]}e^{-\nu_k}\right| \\
& \leq \lambda\epsilon+\epsilon V+\lambda\epsilon.
\end{align*}

The proof for for the integral involving $j'$ follows the same lines, the only 
difference being to control 
\begin{equation*}
\left|\int_{]0,\infty[\setminus [x_k-\delta,x_k+\delta]}j'e^{-\nu_k}\right|.
\end{equation*}
To this end we use the definition of $\nu_k$ and write 
\begin{equation*}
\left|\int_{]0,\infty[\setminus [x_k-\delta,x_k+\delta]}j'e^{-\nu_k}\right|
\leq \left|\int_{]0,\infty[\setminus [x_k-\delta,x_k+\delta]}\nu_{k}'e^{-\nu_k}\right|
+\left|\int_{]0,\infty[\setminus [x_k-\delta,x_k+\delta]}u_{\mathrm{ref}}'e^{-\nu_k}\right| 
\end{equation*}
The second term is controlled as before. 
The first term on the other hand is less than $2\epsilon$ by 
integration and definition of $\delta$. 
\end{proof}

\subsection{$C^2$ estimates}

We turn to \emph{a priori} estimates on $u_t''$. 
Note that the equation at $t$ may be written 
\[ u_t''=e^{-u_{\mathrm{ref}}}Je^{t(u_{\mathrm{ref}}-u_t)}/P(u_t'). \]
Consider again a fixed primitive $Q$ of $P$. It is strictly increasing 
on $[0,\lambda]$. By the properties of $P$, we may find a positive constant 
$C>0$ such that 
\[ y^{n_1+n_2+1}/C \leq Q(y)-Q(0) \leq Cy^{n_1+n_2+1}, \]
\[ (\lambda-y)^{k+1}/C \leq Q(\lambda)-Q(y) \leq C(\lambda-y)^{k+1} \]
and 
\[ y^{n_1+n_2}(\lambda-y)^k/C \leq P(y) \leq Cy^{n_1+n_2}(\lambda-y)^k. \]
Thanks to the $C^0$-estimates, we may further choose this constant $C$ so that 
\[ 1/C \leq e^{t(u_{\mathrm{ref}}-u_t)} \leq C \]
independently of the value of $t$. 

Using the previous inequalities in reverse order, we get 
\[ C^{-1}\frac{P(u_0')}{P(u_t')} 
	\leq \frac{u_t''}{u_0''}=\frac{P(u_0')}{P(u_t')}e^{t(u_{\mathrm{ref}}-u_t)} 
	\leq C\frac{P(u_0')}{P(u_t')} \]
then 
\[ C^{-3}\frac{(u_0')^{n_1+n_2}(\lambda-u_0')^k}{(u_t')^{n_1+n_2}(\lambda-u_t')^k} 
	\leq \frac{u_t''}{u_0''} \leq 
	C^3\frac{(u_0')^{n_1+n_2}(\lambda-u_0')^k}{(u_t')^{n_1+n_2}(\lambda-u_t')^k} \]
and, using the first two inequalities:
\[ \tilde{C}^{-1}
\leq \frac{u_t''}{u_0''} \left(\frac{Q(u_0')-Q(0)}{Q(u_t')-Q(0)}\right)^{\frac{-n_1-n_2}{n_1+n_2+1}} 
\left(\frac{Q(\lambda)-Q(u_0')}{Q(\lambda)-Q(u_t')}\right)^{\frac{-k}{k+1}} \leq 
\tilde{C} \]
where $\tilde{C}=C^{3+\frac{2n_1+2n_2}{n_1+n_2+1}+\frac{2k}{k+1}}$.

We now remember the integral equation  
$Q(u_t'(x))-Q(0) = \int_0^xe^{-u_{\mathrm{ref}}}Je^{t(u_{\mathrm{ref}}-u_t)}$
again and deduce  
\[ C^{-1}(Q(u_0'(x))-Q(0)) 	\leq Q(u_t'(x))-Q(0) \leq C(Q(u_0'(x))-Q(0)) \]
and similarly
\[ C^{-1}(Q(\lambda)-Q(u_0'(x))) 	\leq Q(\lambda)-Q(u_t'(x)) \leq C(Q(\lambda)-Q(u_0'(x))). \]
Putting everything together yields the final estimate comparing $u_t''$ and $u_0''$:
\[ C^{-3-\frac{3n_1+3n_2}{n_1+n_2+1}-\frac{3k}{k+1}} 
	\leq \frac{u_t''}{u_0''} 
	\leq C^{3+\frac{3n_1+3n_2}{n_1+n_2+1}+\frac{3k}{k+1}}. \]
	
\subsection{Closedness}

Assume now that $t_j\in I$, $t_j\rightarrow t$, and we have uniform 
$C^0$ and $C^2$ estimates on $u_{t_j}$ as obtained in other sections
(note that $C^1$ estimates are immediate in view of the restriction 
$u_{t_j}'(\mathbb{R})=]-\lambda,\lambda[$). 
Using Arzela-Ascoli, we obtain a limit function $u_t$ which is $C^1$, 
with locally uniform convergence of $u_{t_j}$ to $u_t$ and of 
$u_{t_j}'$ to $u_t'$. As a consequence, we also know that $u_t$ is an 
even function and that $u_t(x)-\lambda |x|$ is bounded. It remains to 
check that $u_t$ is $C^2$.  Using the equation we have  
$u_{t_j}''=e^{-(t_ju_{t_j}+(1-t_j)u_{\mathrm{ref}})}J/P(u_{t_j}')$. 
Combined with the fact that $u_{t_j}'/u_0'$ is uniformly bounded (this 
follows from the same techniques as the $C^2$ estimates), we obtain that 
$u_{t_j}''$ converges locally uniformly on $\mathbb{R}\setminus \{0\}$ to 
$e^{-(tu_t+(1-t)u_{\mathrm{ref}})}J/P(u_t')$, hence $u_t$ is $C^2$ on 
$\mathbb{R}\setminus \{0\}$ with this same second derivative. 
To conclude, it remains to check that $u_t''$ admits a limit at $0$. 

Note that $u_t$ still satisfies the integral equation 
\[ Q(u_t')-Q(0)=\int_0^xe^{-(tu_t+(1-t)u_{\mathrm{ref}})}J \]
outside $0$.
The polynomial $P$ and $Q$ have the following behavior at $0$:
$P(y)\simeq y^{n_1+n_2}P^{(n_1+n_2)}(0)/(n_1+n_2)!$ and 
$Q(y)-Q(0)\simeq y^{1+n_1+n_2}P^{(n_1+n_2)}(0)/(1+n_1+n_2)!$. 
As a consequence, from the integral equation we have at $0$,
\[ (u_t')^{1+n_1+n_2}\simeq e^{-(tu_t+(1-t)u_{\mathrm{ref}})(0)}(n_1+n_2)!x^{1+n_1+n_2}/P^{(n_1+n_2)}(0) \]
hence $u_t''$ does admit a limit at $0$, hence $u_t$ is $C^2$
on $\mathbb{R}$ with 
\[ u_t''(0)=\left(e^{-(tu_t+(1-t)u_{\mathrm{ref}})(0)}(n_1+n_2)!/P^{(n_1+n_2)}(0)\right)^{1/(1+n_1+n_2)} \]
This ends the proofs of Proposition~\ref{prop:cont} and Theorem~\ref{thm:ODE}. 

\section{Construction of an asymptotically Ricci flat metric}
\label{sec:constr-an-asympt}

Let $G/H$ be an indecomposable rank two (complex) symmetric space. We use the notations introduced in Section~\ref{sec:setup}.
We introduce the three constants $b$, $a_0$ and $b_1$ defined by 
\[ b=2A_2/n \qquad \qquad a_0\tilde{\alpha}_1=b_1\alpha_1+b\tilde{\alpha}_2. \]

\subsection{Approximate solution near $D_2$}
\label{sec:appr-solut-near}

Near (the open $G$-orbit of) $D_2$, that is, when $\alpha_2\rightarrow \infty$ and $\alpha_1$ is bounded, 
we use a Tian-Yau like ansatz. 
We define a potential
\[ \rhodeux=\exp \beta, \quad \beta=b\ta_2+\psi(\alpha_1), \]
where $b$ is the constant defined by $b=2A_2/n$, 
and $\psi$ is a solution to a positive Kähler-Einstein equation on the open orbit of $D_2$, 
with asymptotic behavior imposed by the Ricci flat equation on $G/H$ as we will check. 
More precisely, we assume that the function defined by $u=n\psi+C$, 
where $C=-\ln(2^{n-2-m_{\alpha_1}-m_{2\alpha_1}}b^2n^{1-n})$, is a smooth, strictly convex, even 
solution to the equation 
\begin{equation}
\label{eqn:eqn-psi} 
u''P_{DH}(2A_2\ta_2+u'\alpha_1)=e^{-u}J 
\end{equation}
with $J(x)=\sinh^{m_{\alpha_1}}(x)\sinh^{2m_{\alpha_1}}(2x)$ 
and $P_{DH}$ is the Duistermaat-Heckman polynomial for $G/H$. 
Furthermore, we assume that the function $u$ satisfies the condition
\begin{equation}
u(x)-|nb_1x|=O(1).
\end{equation}
Now, not only are we in a position to apply Theorem~\ref{thm:ODE} to check when such a function 
exists, but one can check that we are exactly in the example of situation described in 
Section~\ref{sec:facets}. The function $u$, if it exists, is thus the potential of a singular 
Kähler-Einstein metric on the colored $\mathbb{Q}$-Fano compactification of the open orbit in $D_2$, 
which is some $G$-equivariant blow-down of $D_2$.  
It follows from Section~\ref{sec:checks} that it is possible to find such a function $\psi$ in all 
cases except when $G=G_2$, in which case only one choice of ordering of the roots allows the 
function $\psi$ to exist. 

Let us check that this gives indeed an asymptotic solution of the Monge-Ampère equation: we obtain
\begin{align*}
  d^2\rhodeux & = \rhodeux \big( d^2\beta+(d\beta)^2 \big) \\
           & = \rhodeux \big( \psi''(\alpha_1) \alpha_1^2 + (b\ta_2+\psi'(\alpha_1)\alpha_1)^2 \big).
\end{align*}
Therefore, using equation \eqref{eqn:eqn-psi},
\begin{align*}
  \det(d^2\rhodeux) \prod_{\alpha\in \rRoots^+}\langle\alpha,d\rhodeux\rangle^{m_\alpha} & 
  = (\rhodeux)^n b^2\psi'' \prod_{\alpha\in \rRoots^+}\langle b\ta_2+\psi'\alpha_1,\alpha\rangle^{m_{\alpha}} \\
  & = 2^{m_{\alpha_1}+m_{2\alpha_1}+2-n}e^{nb\ta_2} \sinh^{m_{\alpha_1}}(\alpha_1) \sinh^{m_{2\alpha_1}}(2\alpha_1).
\end{align*}
Thanks to the symmetry of the root system $\ta_2\pm \lambda\alpha_1$, we have
\[ \sum_{\alpha\in \rRoots^+,\alpha_1\nmid\alpha} m_\alpha \alpha = 2 A_2 \ta_2 = nb \ta_2, \]
and it follows that
\begin{align*}
  \det(d^2\rhodeux) \prod_{\alpha\in \rRoots^+}\langle\alpha,d\rhodeux\rangle^{m_\alpha}
    & = \sinh^{m_{\alpha_1}}(\alpha_1) \sinh^{m_{2\alpha_1}}(2\alpha_1) \prod_{\alpha\in \rRoots^+,\alpha_1\nmid\alpha}e^{m_\alpha \alpha}/2 \\
  & = (1+O(e^{-2\alpha_2})) \prod_{\alpha\in \rRoots^+} \sinh^{m_\alpha}(\alpha) ,
\end{align*}
where $O(e^{-2\alpha_2})$ means functions whose all derivatives with respect to $\alpha_1$ or $\alpha_2$ are bounded by $\mathrm{cst.} e^{-2\alpha_2}$.

Rewriting the Ricci flat equation \eqref{eq:12} as
\begin{equation}
  \label{eq:8}
  \opP(\rhodeux) := \ln\det(d^2\rhodeux) + \sum_{\alpha\in \rRoots^+}m_\alpha \big( \ln \langle\alpha,d\rhodeux\rangle - \ln \sinh \alpha \big) = 0,
\end{equation}
we finally conclude that $\rhodeux$ is an approximate solution when $\alpha_2\rightarrow\infty$ in the sense that for all $\ell$ we have
\begin{equation}
  \label{eq:21}
  |\nabla^\ell \opP(\rhodeux)| \leq c_\ell e^{-2\alpha_2}.
\end{equation}
The solution is good near $D_2$ except when we become close to $D_1$ ($\alpha_1\rightarrow\infty$), where we will construct another model in the next section. It is also important to note (as we will see in Section~\ref{sec:solut-kahl-ricci}) that the geometry when we approach $D_2$ is conical,
and in particular the radius in the cone is \[ r\sim\sqrt{\rhodeux}. \] From the inequality $b<2$ (in all root systems except $G_2$, only $\alpha_2$ and $2\alpha_2$ appear in the root system and this implies immediately $b=2A_2/n<2$; in the $G_2$ case, this is also true, see the tables in §~\ref{sec:summary-constants}), it then follows from \eqref{eq:21} that, when $\alpha_1$ remains bounded and $\alpha_2\rightarrow\infty$,
\begin{equation}
  \label{eq:22}
  P(\rhodeux) = O(r^{-4/b}), \quad4/b > 2.
\end{equation}
which is a good initial control. Our aim now is to construct an asymptotic solution near $D_1$ which can be glued to this one in order to extend the control \eqref{eq:22} to a whole neighborhood of infinity.

\subsection{Approximate solution near $D_1$}  

Near $D_1$, we need to find an asymptotic solution with a good enough
control, and to glue it to the Tian-Yau ansatz produced in
Section~\ref{sec:appr-solut-near}.  Note that from
Proposition~\ref{prop:cont-exp}, $\psi$ admits a precise asymptotic
expansion as $x\rightarrow \infty$. In particular, we introduce the constants
$K_1$, $K_2$ and $a_1$ by
\begin{equation}
\psi(x)=b_1x+K_1+K_2e^{-a_1x}+o(e^{-a_1x}).
\end{equation}
Note that the expression of
$a_1$ was given in Section~\ref{sec:KE}, it is
$a_1=(nb_1-m_{\alpha_1}-2m_{2\alpha_1})/(1+m_{\alpha_2}+m_{2\alpha_2})$.

We define $\zeta=-\frac{\langle\alpha_1,\alpha_2\rangle}{\langle\alpha_2,\alpha_2\rangle}$ so that $\ta_1=\alpha_1+\zeta\alpha_2$.

\begin{prop}\label{prop:refin-model-near-1}
  When $\alpha_1\rightarrow\infty$ there is a development
  \begin{equation}
  \rhoun \sim e^{K_1}e^{a_0\ta_1} \big( 1 + \sum_{k\geq1} e^{-a_k\ta_1}R_k(\alpha_2), \big)\label{eq:17}
\end{equation}
  where $R_k$ is an even function of $\alpha_2$, such that:
  \begin{enumerate}
  \item $0<a_1<a_2<\cdots$, and for $i\geq2$ one has $a_i\in a_1\Bbb{N}+2\Bbb{N}$;
  \item for every $k\geq1$, if $\rhoun_k$ is the truncation of the development at order $k$, then
    \begin{equation}
    |\nabla^\ell \opP(\rhoun_k)| \leq C_{k,\ell} e^{-a_k\alpha_1};\label{eq:23}
  \end{equation}
  \item when $\alpha_2\rightarrow\infty$ then $R_k(\alpha_2) = e^{a_k \zeta \alpha_2} (r_k + O(e^{-2\alpha_2}))$, where $O(e^{-2\alpha_2})$ denotes a function whose all derivatives are $O(e^{-2\alpha_2})$.
  \end{enumerate}
\end{prop}

It is important to note that the terms $e^{-a_k\ta_1+\zeta a_k\alpha_2}=e^{-a_k\alpha_1}$ are actually bounded when $\alpha_2\rightarrow\infty$.

The first truncation $\rhoun_1$ and corresponding function $R_1$ plays an important 
role in understanding the geometry of this model. Let us denote the function $R_1$ by $w$. 
It is obtained as a potential of the Stenzel metric on the symmetric space fiber of the open $G$-orbit 
in $D_1$, in the notations of Example~\ref{exa-Stenzel}.
More precisely, the function $w$ is an (even, smooth, strictly convex) solution to the equation 
\[ Cw''(x)(w'(x))^{m_{\alpha_2}+m_{2\alpha_2}}=\sinh^{m_{\alpha_2}}(x)\sinh^{m_{2\alpha_2}}(2x) \]
with 
$C=2^{n-2-m_{\alpha_2}}a_0^2e^{nK_1}|\alpha_2|^{2(m_{\alpha_2}+m_{2\alpha_2})}
\prod_{\alpha_2 \nmid \alpha\in\rRoots^+}\langle\alpha,a_0\ta_1\rangle^{m_{\alpha}}$. 
Such a solution is defined up to an additive constant, and admits an expansion when $x\rightarrow \infty$ 
which by choosing the additive constant is of the form 
\begin{equation}
 w(\alpha_2) = K_2e^{a_1\zeta\alpha_2} \big( 1 + \sum_{k\geq1} w_ke^{-2k\alpha_2} \big),\label{eq:15}
\end{equation}
for some constants $w_k$. 
Note that one verifies easily from the two one variable equations that the constant $K_2$ and $a_1$ 
in the expansion of $w$ are the same as that in the expansion of $\psi$.

\begin{proof}
  If $k=1$ we take $R_1=w$ so that
  \[ \rhoun_1 = e^{K_1}e^{a_0 \ta_1} \big( 1 + e^{-a_1\ta_1}w(\alpha_2) \big). \]
  Then
  \begin{align*}
 d^2\rhoun_1 = e^{K_1}e^{a_0 \ta_1}& \big( (a_0^2+(a_0-a_1)^2e^{-a_1\ta_1})w(\alpha_2) \ta_1^2 \\ & + (a_0-a_1) e^{-a_1\ta_1}w'(\alpha_2) (\ta_1\alpha_2+\alpha_2\ta_1)\\ & +e^{-a_1\ta_1}w''(\alpha_2) \alpha_2^2 \big).
  \end{align*}
In particular one obtains
\begin{equation}
  \label{eq:13}
  \det(d^2\rhoun_1) = e^{2K_1}a_0^2w''(\alpha_2)e^{(2a_0-a_1)\ta_1} \big(1 + \tfrac{(a_0-a_1)^2}{a_0^2}e^{-a_1\ta_1} ( w(\alpha_2)-\tfrac{w'(\alpha_2)^2}{w''(\alpha_2)}) \big).
\end{equation}
On the other hand, one has
\begin{equation}
  \label{eq:14}
  \langle\alpha,d\rhoun_1\rangle = e^{K_1}e^{(a_0-a_1)\ta_1} w'(\alpha_2) \langle\alpha,\alpha_2\rangle \\
\end{equation}
if $\alpha_2 \mid \alpha$, and
\begin{equation}\label{eq:34}
\langle\alpha,d\rhoun_1\rangle =   e^{K_1}a_0\langle\alpha,\ta_1\rangle e^{a_0\ta_1} \big( 1 + e^{-a_1\ta_1} (\tfrac{a_0-a_1}{a_0}w(\alpha_2) + \tfrac{\langle\alpha,\alpha_2\rangle}{a_0\langle\alpha,\ta_1\rangle} w'(\alpha_2)) \big)  
\end{equation}
if $\alpha_2 \nmid \alpha$.

We define an algebra $\cA$ of formal developments
\[ \cA = \big\{ \sum_{k\geq 1} e^{-a_k\ta_1}f_k(\alpha_2) \big\}, \]
where the coefficients $0\neq a_k\in a_1\Bbb{N}+2\Bbb{N}$ and $f_k$ is an even function satisfying, when $\alpha_2\rightarrow\infty$,
\[ f_k(\alpha_2) = e^{a_k\zeta\alpha_2} (A_k + O(e^{-2\alpha_2})), \]
and all the derivatives of $f_k$ satisfy the same development. More generally we define $\cA_\delta\subset \cA$ as the subalgebra of developments with exponents $a_k\geq\delta$, and we observe that
\[ \cA_\delta \cA_{\delta'} \subset \cA_{\delta+\delta'}. \]

With this formalism, putting together \eqref{eq:13}, \eqref{eq:14}, \eqref{eq:34} and \eqref{eq:15}, it follows that
\begin{equation}
  \label{eq:16}
  \opP(\rhoun_1) \in \cA_{a_1}.
\end{equation}
  
  The linearization of equation \eqref{eq:8} is
  \begin{equation}
    Lf = \tr\big( (d^2\rhoun)^{-1}d^2f \big) + \sum_{\alpha\in \rRoots^+}m_\alpha \frac{\langle\alpha,df\rangle}{\langle\alpha,d\rhoun\rangle}.
  \end{equation}
Writing $df=\partial_{\tilde 1}f \ta_1 + \partial_2f \alpha_2$, where $\partial_{\tilde 1}f=\frac{H_{\ta_1}f }{|\ta_1|^2}$ and $\partial_2f=\frac{H_{\alpha_2}f}{|\alpha_2|^2}$, 
we obtain the formula
\begin{multline}
  \label{eq:10}
  Lf = \frac1{e^{K_1}e^{a_0\ta_1}} \Big( e^{a_1\ta_1} \Delta_2f + 
  a_0^{-2} \partial_{\tilde 1}^2f + a_0^{-1} (n-1-d_2) \partial_{\tilde 1}f \\ + 
  a_0^{-2}(a_0-a_1)^2\frac{w(\alpha_2)}{w''(\alpha_2)}\partial_2^2f 
  + O(e^{-a_1\alpha_1}d^2f) + O(e^{-a_1\alpha_1}df)   \Big),
\end{multline}
where the term $O(e^{-a_1\alpha_1}d^2f)$ means terms in the second derivatives of $f$ with coefficients which are $O(e^{-a_1\alpha_1})$ (with all their derivatives with respect to $\ta_1$ or $\alpha_2$); and $\Delta_2$ is the Laplacian on the symmetric space defined by
\begin{equation}
  \Delta_2f = \frac{\partial_2^2f}{w''(\alpha_2)} + (d_2-1) \frac{\partial_2f}{w'(\alpha_2)}.
\end{equation}

Therefore when $\alpha_1\rightarrow\infty$ the leading order term of $L$ is given just by
\[ e^{-K_1}e^{(-a_0+a_1)\ta_1}\Delta_2. \]
From weighted analysis, we know that if $\delta>0$ then the Laplacian
\[ \Delta_2 : C^k_\delta \longrightarrow C^{k-2}_{\delta-\zeta a_1} \]
is surjective, with kernel reduced to the constants.

Now we can correct our first approximate solution $\rhoun_1$ using the linearization of the equation : from \eqref{eq:16}, we have
\[ \opP(\rhoun_1) = e^{-a_1\ta_1} g(\alpha_2) + h, \]
where
\begin{itemize}
\item $g(\alpha_2)$ is an even function satisfying $g(\alpha_2)=e^{a_1 \zeta \alpha_2} (A+O(e^{-2\alpha_2}))$;
\item   $h\in \cA_{a_2}$, where $a_2=\inf(2a_1,2)$.
\end{itemize}
Then we solve the equation
\begin{equation}
\Delta_2f=g\label{eq:32}
\end{equation}
with $f\in C^k_{\zeta a_1}$ for all $k$ ($f$ is well-defined up to a constant); the form of $\Delta_2$ tells us that, maybe after adjusting the constant if $\zeta a_1<2$,
\begin{equation}
 f(\alpha_2) = e^{2a_1 \zeta \alpha_2} (B+O(e^{-2\alpha_2})),\label{eq:33}
\end{equation}
where the term $O(e^{-2\alpha_2})$ again means that all derivatives have the same decay. This is exactly the required expansion so that the function
\begin{equation}
  \rhoun_2 = \rhoun_1 - e^{K_1} e^{(a_0-2a_1)\ta_1} f(\alpha_2),
\end{equation}
has the form expected in the statement of the proposition. If we apply the other terms of the linearization $L$ defined in \eqref{eq:10} to $e^{(a_0-2a_1)\ta_1} f(\alpha_2)$, we obtain a function in $\cA_{a_2}$; the nonlinear terms of $\opP(\rhoun_2)$ also behave well thanks to the multiplication properties in $\cA$, so that one obtains finally
\[ \opP(\rhoun_2) \in \cA_{a_2}. \]
We can iterate this procedure to construct inductively $\rhoun_k$, and this gives the proposition.
\end{proof}

\subsection{The approximate solution}
\label{sec:approximate-solution}

Near the divisor $D_2$ we have the other approximate solution $\rhodeux=e^{b\ta_2+\psi(\alpha_1)}$, 
with $\psi(\alpha_1)$ satisfying the equation \eqref{eqn:eqn-psi}. We have an asymptotic development
\begin{equation}
 \psi(\alpha_1) \sim b_1\alpha_1 + K_1 + \sum_{k\ge1} c_k e^{-a_k\alpha_1}, \quad a_k\in a_1\Bbb{N}+2\Bbb{N}.\label{eq:18}
\end{equation}

\begin{lem}
\label{lem:est-diff}
  If we take for each term $R_k(\alpha_2)$ of the development \eqref{eq:17} the top order term $r_ke^{a_k\zeta \alpha_2}$, then we obtain the same development as in \eqref{eq:18}, that is (formally)
  \[ \exp \sum_{k\ge1} c_k e^{-a_k\alpha_1} = 1+\sum_{k\geq1} r_k e^{-a_k\alpha_1}. \]
  In particular, the difference $\rhodeux-\rhoun$ has a formal development
  \[ \rhodeux - \rhoun \sim e^{K_1}e^{a_0\ta_1} \sum_{k\geq1} e^{-a_k\alpha_1} g_k(\alpha_2),\]
  with each $g_k(\alpha_2)=O(e^{-2\alpha_2})$. 
\end{lem}

The lemma means that each term $e^{b\ta_2-a_k\alpha_1}$ of the development of $\rhodeux$ glues well with the terms of $\rhoun$: one can actually interpret the construction of $\rhoun$ as an extension along $D_1$ of each term of this asymptotic term, so that one obtains an asymptotic solution along $D_1$ at any order.

\begin{proof}
  We can rewrite $\varrho^{(2)}$ in terms of the coordinates $(\tilde\alpha_1,\alpha_2)$ used to construct $\varrho^{(1)}$: since $\tilde\alpha_1=\alpha_1+\zeta\alpha_2$,
  \[ \varrho^{(2)}=\exp\big(a_0\tilde\alpha_1+K_1+\sum_{k\geq1}c_ke^{-a_k(\tilde\alpha_1-\zeta\alpha_2)}\big).\]
  This is by \eqref{eq:21} a formal solution of the equation
  \begin{equation}
 \opP(\varrho^{(2)}) = O(e^{-2\alpha_2}).\label{eq:31}
 \end{equation}
We then just need to check that the top order terms of $\varrho^{(1)}$, that is
\[ \varrho^{(1)}_{top} := e^{K_1+a_0\tilde\alpha_1} \sum_{k\geq1} r_k e^{-a_k(\tilde\alpha_1-\zeta\alpha_2)} \]
also satisfy \eqref{eq:31}, and that the formal solution of \eqref{eq:31} in powers of $e^{\alpha_1}=e^{\tilde\alpha_1-\zeta\alpha_2}$ is unique.

Note \[ \tau=\varrho^{(1)}-\varrho^{(1)}_{top}=e^{K_1+a_0\tilde\alpha_1}\sum_{k\geq1}e^{-a_k\alpha_1}O(e^{-2\alpha_2}), \] 
then it is clear that the contribution of $\tau$ in $\opP(\varrho^{(1)})$ is $O(e^{-2\alpha_2})$, that is
\[ \opP(\varrho^{(1)}) = \opP(\varrho^{(1)}_{top}) + O(e^{-2\alpha_2}). \]
It follows that $\varrho^{(1)}_{top}$ is also a formal solution of \eqref{eq:31}. The uniqueness can be obtained by specializing the construction of the formal development in the proof of Proposition~\ref{prop:refin-model-near-1} to the top order terms and checking that at each step the top order term is uniquely determined: this is true because when we solve \eqref{eq:32} the ambiguity is a constant but the top order term blows up \eqref{eq:33} and is completely determined by the previous top order terms. 
\end{proof}

This now enables to glue together the potentials $\rhodeux$ and $\rhoun$ along a ray $\alpha_1=\eta \alpha_2$ in the following way. We truncate $\rhoun$ to some order $k$ into $\rhoun_k$. We choose a smooth nondecreasing function $\chi$ on $\Bbb{R}$ such that $\chi(t)=0$ if $t\leq0$ and $\chi(t)=1$ if $t\geq1$, and define
\begin{equation}
 \varrho = \chi(\alpha_1-\eta \alpha_2) \rhoun_k + (1-\chi(\alpha_1-\eta \alpha_2)) \rhodeux.\label{eq:19}
\end{equation}
On the transition region $0\leq\alpha_1-\eta \alpha_2\leq1$, we write $\varrho=\rhoun_k+(1-\chi(\alpha_1-\eta\alpha_2))(\rhodeux-\rhoun_k)$. By the lemma, and using the fact that $\chi(\alpha_1-\eta\alpha_2)$ and all its derivatives are bounded, one obtains that, still on the transition region, the linearization $L$ calculated in \eqref{eq:10} satisfies
\[ L(\varrho-\rhoun_k) = O(e^{-2\alpha_2}+e^{(a_1-a_{k+1}) \alpha_1}),  \]
where again the $O(\cdot)$ means a function such that all derivatives with respect to $\alpha_1$ or $\alpha_2$ satisfy the same estimate. The nonlinear terms are even smaller, so we finally get on the transition region
\begin{equation}
 \opP(\varrho) = O(e^{-2\alpha_2}+e^{(a_1-a_{k+1}) \alpha_1}).\label{eq:20}
\end{equation}

\begin{prop}\label{prop:est-P}
  Take $\eta<\zeta(2/b-1)$ and $k$ large enough so that $a_k>a_0(1+\zeta/\eta)$. Then, for $(\alpha_1,\alpha_2)$ outside a large compact set, we have for all $\ell$
  \[ |\nabla^\ell \opP(\varrho)| \leq C_\ell e^{-(1+\varepsilon)\beta}, \quad \beta=b\ta_2+\psi(\alpha_1). \]
\end{prop}
\begin{proof}
  The idea of the proof is simple: near $D_2$ (that is, when $\alpha_2\rightarrow\infty$ while $\alpha_1$ remains bounded) we already have such a control, see \eqref{eq:22}, and therefore the control persists up to the gluing region $0<\alpha_1-\eta \alpha_2<1$ provided that $\eta$ is small enough. On the contrary, if $\eta$ is small then we need a high order control in powers of $e^{-\alpha_1}$ near $D_1$ in order to control up to the transition region: this is provided by Proposition~\ref{prop:refin-model-near-1}.

  More precisely, observe that when $\alpha_1\rightarrow\infty$ one has $\beta= a_0\ta_1+O(1)$. Then:
  \begin{itemize}
  \item on the region $\alpha_1\leq \eta\alpha_2$ then $a_0\ta_1\leq b(\eta/\zeta+1)\alpha_2$ so $e^{-2\alpha_2}=O(e^{-(1+\varepsilon)\beta})$ on this region if $\eta<\zeta(2/b-1)$;
  \item on the region $\alpha_1\geq \eta\alpha_2$ then $a_0\ta_1\leq a_0(1+\zeta/\eta) \alpha_1$ so $e^{(a_1-a_{k+1})\alpha_1}=O(e^{-(1+\varepsilon)\beta})$ on this region if $a_{k+1}-a_1>a_0(1+\zeta/\eta)$.
\end{itemize}
Given the controls \eqref{eq:23} and \eqref{eq:22} near $D_1$ and $D_2$, and the control \eqref{eq:20} in the transition region, the proposition follows.
\end{proof}

We will now modify slightly this function obtained by gluing to make it a well defined $W$-invariant 
smooth and strictly convex function, thus corresponding to a Kähler metric on $G/H$. 
Recall that $\varrho$ coincides with $\rhodeux$ in the region defined by $\alpha_1\leq \eta\alpha_2$ 
and that $\rhodeux$ is invariant under the reflection defined by $\alpha_1$ since $\ta_2$ is 
orthogonal to $\alpha_1$ and $\psi$ is even. Similarly, on the region defined by 
$\alpha_1\geq \eta\alpha_2+1$, $\varrho$ coincides with $\rhoun_k$ which is invariant 
under the reflection with respect to $\alpha_2$. 
From this we deduce that the $W$-invariant function, still denoted by $\varrho$, 
whose restriction to the positive Weyl chamber is $\varrho$, is smooth outside of a 
large enough compact set. 

Let us now show that $\varrho$ is strictly convex outside of a large enough compact set. 
Note that $\rhodeux$ is strictly convex by construction. 
We restrict to a region of the form $\{\alpha_1\geq \epsilon \alpha_2\geq 0\}$ for some $\epsilon>0$. 
In restriction to such a region, we have $e^{a_0\ta_1-a_k\alpha_1}=o(e^{a_0\ta_1-a_{k-1}\alpha_1})=o(e^{a_0\ta_1})$ 
at infinity, for $k\geq 2$.
For simplicity, we identify $\chi$ with the composition $\chi(\alpha_1-\eta \alpha_2)$, and compute 
\[ d^2\varrho=d^2\rhoun_k +(1-\chi)(d^2\rhodeux-d^2\rhoun_k)-2d\chi d(\rhodeux-\rhoun_k)-(\rhodeux-\rhoun_k)d^2\chi \]
We have at least $\rhodeux-\rhoun_k=O(e^{a_0\ta_1-a_2\alpha_1})$, 
and the derivatives of $\chi$ are bounded, hence the two last terms above are of this order. 
On the other hand, $d^2\rhoun_k +(1-\chi)(d^2\rhodeux-d^2\rhoun_k)$ is 
\begin{multline*}
e^{K_1}e^{a_0\ta_1}\big( (a_0^2+O(e^{-a_1\alpha_1}))\ta_1^2 + O(e^{-a_1\alpha_1})(\ta_1\alpha_2+\alpha_2\ta_1) 
\\ +e^{-a_1\ta_1}(\chi w''(\alpha_2)+(1-\chi)K_2a_1^2\zeta^2e^{a_1\zeta\alpha_2})\alpha_2^2 \big).
\end{multline*}
We may now conclude, in view of the properties of $w$ (which is strictly convex and such that 
$w(\alpha_2)=K_2e^{a_1\zeta\alpha_2}(1+O(e^{-2\alpha_2})$), that the dominant term of $\det(d^2\varrho)$ 
at infinity is strictly positive. 
Furthermore, the dominant term for the matrix itself is $e^{K_1}e^{a_0\ta_1}a_0^2\ta_1^2$, which is 
semi-positive, hence we may find a compact set outside of which the function $\varrho$ is strictly convex.  

We finally glue in an arbitrary smooth, $W$-invariant, strictly convex function on the compact 
set where $\varrho$ is not well-behaved as follows.
Let $M\in \mathbb{R}$ and consider the function 
\begin{equation*}
\varrho_{\mathrm{int}}:=M+\ln\sum_{w\in W}e^{w\cdot\alpha_1}.
\end{equation*}
It is a smooth, $W$-invariant and strictly convex function on $\mathfrak{a}$, 
and we may assume, by choosing $M$ large enough, that 
$\varrho_{\mathrm{int}}\geq \varrho$ 
on the compact set where it is not well-behaved.
Now consider the function defined by $\sup(\varrho_{\mathrm{int}},\varrho)$.
It is a convex function, smooth and strictly convex outside of the set where 
$\varrho_{\mathrm{int}}$ and $\varrho$ coincide, which is compact by comparison 
of the growth rates.
We finally choose an approximation of this supremum which is 
smooth, strictly convex, and equal to $\varrho$ outside of a compact set containing 
the contact set of $\varrho$ and $\varrho_{\mathrm{int}}$. 
This is possible using for example \cite{Gho02}. 
This final function provides the desired asymptotic solution, and we still denote 
it by $\varrho$ in the following. 

\section{Solution to the Kähler-Ricci flat equation}
\label{sec:solut-kahl-ricci}

\subsection{The asymptotic metric}

Let $(l_1,l_2)$ denote the basis of $\fa$ which is dual to the basis of restricted roots $(\alpha_1,\alpha_2)$.
We use the notation $\Roots_s$ to denote the roots of $G$ which are not stable under $\sigma$, 
and let $\ha_r=\ha-\sigma(\ha)$ denote the restricted root associated to $\ha\in\Roots_s$. 
Recall that with this convention, $\ha_r|_{\fa}=2\ha|_{\fa}$.
For each $\ha\in  \Roots^+$ denote $\mu_{\ha}=e_{\ha}+\sigma(e_{\ha})$, where $(\ha^\vee,e_{\ha},e_{-\ha}=-\theta(e_{\ha}))$ is a $sl_2$-triple. 
(Here $\ha^\vee$ is defined by $\ha^\vee=\frac{2H_{\ha}}{|\ha|^2}$).

Then we can parametrize the symmetric space by
\begin{equation}
(z_1,z_2,(z_{\ha})_{\ha\in \Roots_s^+}) \mapsto \exp( \sum_{\ha\in \Roots_s^+} z_{\ha} \mu_{\ha} ) \exp(z_1 l_1 + z_2 l_2) H\label{eq:24}
\end{equation}
which is a local biholomorphism when $\Re(z_1) l_1+\Re(z_2)l_2$ belongs to the regular part of $\fa$.

Using the forms $\omega_{a\bar b}=\frac i2 dz_a\wedge d\bar z_b$ for $a,b\in \{1,2\}$, and $\omega_{\ha\bar \ha}=\frac i2 dz_{\ha}\wedge d\bar z_{\ha}$, then a $K$-invariant Kähler potential is given by a function $\varrho(x_1,x_2)$ on $\fa$, and it follows from \cite[Corollary 2.11]{DelHoro} that the Kähler form on the symmetric space $G/H$ is given along the regular part of $A=\exp\fa$ by
\begin{equation}
 \sum_{a,b\in\{1,2\}} d^2\varrho(l_a,l_b) \omega_{a\bar b}
                + 2 \sum_{\ha\in \Roots_s^+} \tanh(\ha) \frac{\langle d\varrho, \ha_r\rangle}{|\ha|^2} \omega_{\ha\bar\ha}.\label{eq:25}
\end{equation}
The parametrization \eqref{eq:24} is slightly different from that in \cite{DelHoro} which explains that the formula is not exactly the same: in \eqref{eq:24} we choose the coordinates $(z_\alpha)$ given by the group action on $e^{z_1l_1+z_2l_2}H\in A$; this choice still makes sense on the compactification (when $z_1$ or $z_2$ go to infinity), so our formulas will be meaningful also on $\bar{A}\cap D_1$ and $\bar{A}\cap D_2$. 

Note that with this normalization, the restriction to $A$ of the metric $g$ corresponding to the the Kähler form \eqref{eq:25} is given in coordinates $(x_1,x_2)$ by
\begin{equation}
  g|_A = \Hess \varrho. 
\end{equation}

With these formulas at hand, we can now give the asymptotic behavior of the metric at infinity. We define as in 
Section~\ref{sec:appr-solut-near} the function
\[ \beta = b\ta_2+\psi(\alpha_1). \]
Then near $D_2$, that is when $\alpha_2\rightarrow\infty$, the potential $e^\beta$ leads to a metric
\begin{equation}
  \label{eq:27}
  g_2 = e^\beta \Big( |d\beta|^2 + \psi''(\alpha_1) |\alpha_1|^2 + \sum_{\ha\in \Roots_s^+} \frac2{|\ha|^2} \tanh(\ha) \langle d\beta,\ha_r\rangle |dz_{\ha}|^2 \Big).
\end{equation}
Since $d\beta=b\ta_2+\psi'(\alpha_1)\alpha_1$ and $\psi'(\alpha_1)>0$, we have $\langle d\beta,\ha_r\rangle>0$ for all $\ha_r\in \Roots_r^+$ and all values of $\alpha_1$. Therefore, the formula \eqref{eq:27} is an asymptotically conical metric with radius $r=2e^{\beta/2}$ when we approach $D_2$, that is when $\alpha_2\rightarrow\infty$ while $\alpha_1$ remains bounded.

We now pass to the behavior near $D_1$ of the metric given by the principal term of the potential, $\rhoun_1=\exp(a_0\ta_1+e^{-a_1\ta_1} w(\alpha_2))$. The same calculation now gives
\[\begin{split}
  g_1 = & \rhoun_1 \Big( \big|(a_0-a_1e^{-a_1\ta_1}w)\ta_1+e^{-a_1\ta_1} w'\alpha_2\big|^2 \\ & + a_1^2e^{-a_1\ta_1} w |\ta_1|^2 + e^{-a_1\ta_1}w'' |\alpha_2|^2  \\
  & + \sum_{\ha\in \Roots_s^+} \frac2{|\ha|^2} \tanh(\ha) \big\langle(a_0-a_1e^{-a_1\ta_1}w)\ta_1+e^{-a_1\ta_1}w'\alpha_2,\ha_r\big\rangle |dz_{\ha}|^2 \Big).
\end{split}\]
Splitting the sum into roots such that $\ha_r$ is a multiple of $\alpha_2$ and other roots, we write the principal part as
\begin{multline}\label{eq:29}
g_{\mo} = \rhoun_1 \Big( a_0^2|\ta_1|^2 + 2a_0 \sum_{\alpha_2 \nmid \ha_r} \tanh(\ha) \tfrac{\langle\ta_1,\ha_r\rangle}{|\ha|^2} |dz_{\ha}|^2 \\
+ e^{-a_1\ta_1} \big( w'' |\alpha_2|^2 + 2w' \sum_{\alpha_2 \mid \ha_r} \tanh(\ha) \frac{\langle\alpha_2,\ha_r\rangle}{|\ha|^2} |dz_{\ha}|^2 \big) \Big).
\end{multline}
Then, using that $w(\alpha_2)=O(e^{a_1\zeta\alpha_2})$ when $\alpha_2\rightarrow\infty$ and therefore $e^{-a_1\ta_1}w(\alpha_2)=O(e^{a_1\alpha_1})$, with the same for the derivatives with respect to $\alpha_2$, we obtain 
\begin{equation}
  | g_1 - g_{\mo} |_{g_{\mo}} = O(e^{-a_1\alpha_1}).
\end{equation}
Therefore the equation \eqref{eq:29} gives the asymptotics of $g_1$ when $\alpha_1\rightarrow\infty$.

The metric $g_a$ is not exactly asymptotically conical since $g_a/\rhoun_1$ collapses along the directions given by the action of $H_{\alpha_2}$ and the $\mu_{\ha}$, when $\alpha_2\mid\ha_r$, that is along the directions of the fibers of the fibration $D_1\rightarrow G/P_1$, which are isomorphic to the symmetric space $X_1$; and the metric \[ w'' |\alpha_2|^2 + 4w' \sum_{\alpha=k\alpha_2\in R^+} k |dz_\alpha|^2 \] is the asymptotically conical Kähler Ricci flat metric on $X_1$.

Of course it is important to note that on the regular part of the Weyl chamber, the formulas \eqref{eq:27} and \eqref{eq:29} give the same asymptotic behavior, since then $\beta= a_0\ta_1+K_1+K_2e^{-a_1\alpha_1}+O(e^{-2a_1\alpha_1})$ and the asymptotics of $\psi''(\alpha_1)$ and $e^{-a_1\ta_1}w''(\alpha_2)\sim a_1^2e^{-a_1\alpha_1}$ match, so we again obtain
\[ |g_2-g_{\mo}|_{g_{\mo}} = O(e^{-a_1\alpha_1}). \]
Our definitive initial metric $g_0$ derives from the potential $\varrho$ obtained by gluing the potential $\rhodeux=e^\beta$ with the potential $\rhoun_k$ for some large $k$ as described in §~\ref{sec:approximate-solution}. Of course $g_0$ is also asymptotic to $g_{\mo}$ when $\alpha_1$ goes to infinity:
\[ |g_0-g_{\mo}|_{g_{\mo}} = O(e^{-a_1\alpha_1}). \]

If we now replace $\rhoun_1$ by the potential $\rhoun_k$ from Proposition~\ref{prop:refin-model-near-1}, leading to the potential given by \eqref{eq:19}, then we of course get a higher order coincidence between $dd^C\rhoun_k$ and $dd^C\rhodeux$: more precisely, from Lemma~\ref{lem:est-diff} we have the estimate (also true for the derivatives):
\[ \rhodeux - \rhoun_k = O\big(e^{a_0\ta_1} (e^{-a_1\alpha_1-2\alpha_2}+e^{-a_{k+1}\alpha_1}) \big), \]
we obtain
\[ |dd^C\rhodeux - dd^C\rhoun_k|_{g_{\mo}} = O\big( e^{-2\alpha_2} + e^{(a_1-a_{k+1})\alpha_1} \big), \]
and more generally
\[ |\nabla^\ell(dd^C\rhodeux - dd^C\rhoun_k)|_{g_{\mo}} = O\big( e^{\frac\ell2 \alpha_1}(e^{-2\alpha_2} + e^{(a_1-a_{k+1})\alpha_1}) \big). \]

The Ricci form is given by
\[ \Ric = - \frac12 dd^C \opP(\varrho). \]
From Proposition~\ref{prop:est-P}, we obtain:
\begin{prop}\label{prop:est-Ric}
  Given any integer $\ell_0$, if the coefficient $\eta$ defining the transition region is small enough and $k$ is large enough, then for the metric $g_0$ coming from the potential $\varrho$ given by \eqref{eq:19}, one has for all $\ell\leq\ell_0$
  \[ |\nabla^\ell \opP(\varrho)|_{g_0} \leq C_\ell e^{(-1-\varepsilon-\frac\ell2)\beta}. \]
In particular, for $\ell\leq \ell_0-2$, one has
\[ |\nabla^\ell\Ric|_{g_0} \leq C_\ell e^{(-2-\varepsilon-\frac\ell2)\beta}. \]
\end{prop}
\begin{proof}
  The proof is similar to that of Proposition~\ref{prop:est-P}, the difference being that we now calculate the derivatives with respect to the metric $g_0$, hence the weight $e^{-\frac \beta2}$ for each derivative, and additionally $e^{\frac{a_1\alpha_1}2}$ for derivatives in the direction of $H_{\alpha_2}$ when we go to $D_1$. Because of this last weight, the proposition is not an immediate consequence of Proposition~\ref{prop:est-P}, but the scheme of proof is the same : we check what happens in the various regions.
  \begin{itemize}
  \item In the direction of $D_2$ ($\alpha_2\rightarrow\infty$, $\alpha_1$ bounded), we have $\opP(\varrho)=O(e^{-2\alpha_2})$ (with the same estimates for the derivatives), and therefore, given the geometry of the metric,
    \[ |\nabla^\ell \opP(\varrho)|_{g_0} = O(e^{-2\alpha_2-\frac \ell2 \beta}) . \]
  \item In the direction of $D_1$ ($\alpha_1\rightarrow\infty$, this includes the transition region), we have $\opP(\rhoun_k)=O(e^{-a_k\alpha_1})$; here, because of the geometry of the metric, each derivative in the $H_{\alpha_2}$ direction comes with a weight $e^{\frac{a_1}2\alpha_1}$, and therefore
    \[ |\nabla^\ell \opP(\rhoun_k)|_{g_0} = O(e^{(-a_k+\frac\ell2)\alpha_1-\frac \ell2 \beta}). \]
  (Recall $\beta\sim a_0\ta_1$ in this direction).
\end{itemize}
If $\ell_0$ is given, we can take $k$ large enough in order to have $-a_k+\frac{\ell_0}2$ as negative as we want, and we then proceed as in the proof of Proposition~\ref{prop:est-P}.
\end{proof}
It is clear from the proof that it is impossible to control all the derivatives of the Ricci tensor when one goes to $D_1$, because of the collapsed directions. This is usually remedied in the literature by using weighted spaces with two weights, one of the weights taking care of the collapsed directions. We will use another approach and just state the bounds in the proposition in order to control the geometry at infinity of $g_0$.

\subsection{The Ricci flat Kähler metric}

\begin{lem}\label{lem:atlas}
  Fix $\ell_0$ and then $g_0$ as in Proposition~\ref{prop:est-Ric}. If $a_1\leq a_0$ then the injectivity radius of $g_0$ is bounded below, and $g_0$ admits a $C^{\ell_0-1,\alpha}$ atlas; in particular  the curvature of $g_0$ is bounded in $C^{\ell_0-3,\alpha}$.
\end{lem}
By a $C^{k,\alpha}$ atlas we mean the local existence of holomorphic diffeomorphisms with a ball $B\subset \Bbb{C}^n$ such that $C^{-1}g_{\Bbb{C}^n} \leq g_0 \leq C g_{\Bbb{C}^n}$ and $\|g_0\|_{C^{k,\alpha}(B)} \leq C$. The notion of quasi-atlas is similar but the diffeomorphisms can be only local diffeomorphisms: this is used in Tian-Yau \cite{TiaYau90} but here we need only the notion of atlas.
\begin{proof}
  This follows immediately from the model \eqref{eq:29} for the metric at infinity : if $a_1>a_0$ then there is a collapsing in the directions of the fibers $X_1$ when $\ta_1\rightarrow\infty$, and it follows that the injectivity radius goes to zero since it behaves like that $e^{\frac{a_0-a_1}2\ta_1}\inj_{X_1}$. But if $a_1\leq a_0$ all directions blow up or at least remain bounded below when one goes to infinity, so the injectivity radius stays bounded below.

  A lower bound on the injectivity radius and the bound on $\ell_0-2$ derivatives of Ricci (Proposition~\ref{prop:est-Ric}) gives a lower bound on the $C^{\ell_0-1,\alpha}$ harmonic radius of $g_0$, which gives a $C^{\ell_0-3,\alpha}$ bound on the curvature of $g_0$. From this it is easy to pass to a $C^{\ell_0-3,\alpha}$ atlas, see for example \cite{TiaYau90}.
\end{proof}

We produce the Kähler Ricci flat metric by using the Tian-Yau theorem \cite{TiaYau91} in the version written in the PhD Thesis of Hein \cite[Proposition 4.1]{Hei10}. The hypothesis on the initial metric $g_0$ are:
\begin{enumerate}
\item the existence of a $C^{3,\alpha}$ quasi-atlas, which follows from Lemma~\ref{lem:atlas} with $\ell_0\geq4$;
\item an initial Ricci potential $f\in C^{2,\alpha}$ decaying as $O(r^{-2-\varepsilon})$: this follows from Proposition~\ref{prop:est-Ric} with $\ell_0\geq 3$;
\item the condition SOB($n$): there exists a point $x_0$ and $C\geq 1$ such that if we note $r(x)$ the distance to $x_0$, then the annuli $A(x_0,s,t)$ are connected for all $t>s\geq C$, $\Vol(B(x_0,s))\leq Cs^n$ for all $s\geq C$, and $\Vol(B(x,(1-C^{-1})r(x))) \geq C^{-1} r(x)^n$ and $\Ric(x)\geq -Cr(x)^{-2}$: all these conditions are clear given our explicit model.
\end{enumerate}
The theorem of Hein now produces a Kähler Ricci flat metric $\omega_0+dd^Cu$ with $u\in C^{4,\bar \alpha}$ for some $\bar \alpha\leq \alpha$. Therefore this metric has the same asymptotic cone than $\omega_0$, and the theorem is proved.

\begin{rem}
  The function $e^\beta\sim \frac{r^2}4$ gives the asymptotic potential at infinity, which implies that $\Delta(e^\beta) \sim n$ (including when one goes to $D_1$, that is in the directions where there is collapsing). The functions $e^{\delta \beta}$ are then well suited to barrier arguments, and one can then prove that, if we write the Ricci flat metric $\omega_0+dd^Cu$, then one has actually $u=O(e^{-\varepsilon\beta})$, see \cite[§ 4.5]{Hei10}.
\end{rem}

\section{Summary of constants}
\label{sec:summary-constants}
We gather in Table~\ref{tab:constants} the expression of notable constants in terms of the multiplicities 
in the restricted root system, as well as the indexing of simple restricted roots.
Recall that the dimension $n$ of $X$ is $n=2+\sum_{\alpha\in\rRoots^+} m_{\alpha}$, 
that the dimension of the fibers of the facets are 
$\dim(X_1)=1+m_{\alpha_2}+m_{2\alpha_2}$ and 
$\dim(X_2)=1+m_{\alpha_1}+m_{2\alpha_1}$. 
The coefficients of $\varpi=A_1\alpha_1+A_2\alpha_2$ were computed in 
Section~\ref{sec:checks} and are recalled in the table.
For the Tian-Yau ansatz, we introduced $b=2A_2/n$, then 
set $a_0=b |\ta_2|^2/\langle \ta_1,\ta_2 \rangle$ 
and $b_1=b \langle \ta_2,\alpha_2\rangle / \langle \alpha_1,\alpha_2\rangle$.
Finally, the constant $a_1$ appeared in the expansions, 
and is equal to 
$nb_1-m_{\alpha_1}-2m_{2\alpha_1}/(1+m_{\alpha_2}+m_{2\alpha_2})$.

We also include in the table when the condition $a_1\leq a_0$ is satisfied, 
and when the positive Kähler-Einstein metric needed exists on $\check{D}_2$. 
Note that we consider only the values of multiplicities that appear in symmetric spaces. 

\begin{table}
\begin{equation*}
\begin{array}{cccccc}
\toprule
 & A_2 & B(C)_2~(\alpha_1=\alpha) & B(C)_2~(\alpha_1=\beta) & G_2~(\alpha_1=\alpha) & G_2~(\alpha_1=\beta) \\
\midrule
\left<\alpha_1,\alpha_2\right> & -1/2 & -1 & -1 & -3/2 & -3/2\\
\left<\alpha_1,\alpha_1\right> & 1 & 2 &1  & 3 & 1\\
\left<\alpha_2,\alpha_2\right> & 1 & 1 &2 & 1 & 3\\
n & 2+3m & 2(1+m_1+m_2+m_3) & 2(1+m_1+m_2+m_3) & 2+6m & 2+6m \\
\dim(X_1) & 1+m & 1+m_2+m_3 & 1+m_1 & 1+m & 1+m \\
\dim(X_2) & 1+m & 1+m_1 & 1+m_2+m_3 & 1+m & 1+m \\
A_1 & m & m_1+m_2/2+m_3 & m_1+m_2+2m_3 & 3m & 5m \\
A_2 & m & m_1+m_2+2m_3 & m_1+\frac{m_2}{2}+m_3 & 5m & 3m \\
\midrule
b & \frac{2m}{2+3m} & \frac{m_1+m_2+2m_3}{1+m_1+m_2+m_3} & \frac{2m_1+m_2+2m_3}{2(1+m_1+m_2+m_3)} & \frac{5m}{1+3m} & \frac{3m}{1+3m} \\
b_1 & 3b/2 & b/2 & b & b/6 & b/2 \\
a_0 & 2b & b & 2b & 2b/3 & 2b \\
a_1 & \frac{2m}{1+m} & \frac{m_2+2m_3}{1+m_2+m_3} & \frac{2m_1}{1+m_1} & \frac{2m}{3(1+m)} & \frac{2m}{3(1+m)} \\
\midrule
a_1\leq a_0 & \mathrm{false} & m_3\leq 1 & m_2(m_1-1)\leq 2m_3 & \mathrm{true} & \mathrm{true} \\
\midrule
\check{D}_2~\mathrm{KE?} & \mathrm{true} & \mathrm{true} & \mathrm{true} & \mathrm{true} & \mathrm{false} \\
\bottomrule
\end{array}
\end{equation*}
\caption{Notable constants and conditions}
\label{tab:constants} 
\end{table}

\bibliographystyle{alpha}
\bibliography{KRFSS}

\end{document}